\documentclass{article}
\pdfoutput=1
\usepackage{amsmath}
\usepackage{amsfonts}
\usepackage{amssymb}
\usepackage{graphicx}%
\usepackage{tikz}
\usetikzlibrary{decorations.markings}
\usetikzlibrary{arrows}

\newtheorem{theorem}{Theorem}

\newtheorem{corollary}[theorem]{Corollary}

\newtheorem{definition}[theorem]{Definition}

\newtheorem{lemma}[theorem]{Lemma}

\newtheorem{proposition}[theorem]{Proposition}

\newenvironment{proof}[1][Proof]{\noindent\textbf{#1.} }{\ \rule{0.5em}{0.5em}}

\begin{document}

\title{Link mutations and Goeritz matrices}
\author{Lorenzo Traldi\\Lafayette College\\Easton Pennsylvania 18042}
\date{ }
\maketitle

\begin{abstract}
Extending theorems of J.\ E.\ Greene [Invent.\ Math.\ \textbf{192} (2013), 717-750] and A.\ S.\ Lipson [Enseign.\ Math.\ (2) \textbf{36} (1990), 93-114], we prove that the equivalence class of a classical link $L$ under mutation is determined by Goeritz matrices associated to diagrams of $L$. 

AMS Subject Classification: 57M25

Keywords: Goeritz matrix, mutation, link

\end{abstract}

\date{}

\section{Introduction}
A (classical, unoriented) \emph{link} is a union $L=K_{1}\cup \dots \cup K_{\mu(L) }$ of
pairwise disjoint, piecewise smooth closed curves $K_i$ in $\mathbb{R}^{3}$. Each $K_{i}$
is a \emph{knot}, and these knots are the \emph{components} of $L$; $\mu(L)$ is the \emph{component number} of $L$. Two
links are \emph{ambient isotopic} or \emph{of the same type} if there is a piecewise smooth, orientation-preserving
autohomeomorphism of $\mathbb{R}^{3}$, which maps one link to the other. The ambient isotopy class of a link is a \emph{link type}. It is common to blur the distinction between individual links and link types, but we are careful about this distinction in the present paper. We use $[L]$ to denote the link type of $L$. 

An \emph{oriented link} is given with preferred orientations of its components, and an ambient isotopy between two oriented links must respect these orientations. Most of our discussion involves unoriented links, which we usually call ``links''; when we discuss oriented links, we always include the adjective ``oriented.''

If $P$ is a plane and $L$ is a link contained in one of the components of $\mathbb{R}^3-P$ then $L$ may be described using a \emph{diagram} in $P$. A diagram is obtained in two steps. First $L$ is projected into $P$; the only allowable singularities of the projection are a finite number of transverse double points, called \emph{crossings}. Second, at each crossing the undercrossing strand (the strand closer to $P$) is indicated by removing two very short segments, one on each side of the projected image of the crossing. A fundamental theorem of classical knot theory is that ambient isotopy is completely described by three simple moves on diagrams, the Reidemeister moves~\cite{R}.

In this paper we are concerned with two equivalence relations on links and two similar equivalence relations on oriented links; the relations are all coarser than ambient isotopy. Two of these equivalence relations involve a matrix associated to a link diagram, which has been studied by many researchers since it was discussed by Goeritz almost 100 years ago~\cite{G}. Despite the attention of these researchers, the precise knot-theoretic significance of the Goeritz matrix has not been established. That is to say, we do not know how to determine whether or not two link types have diagrams with the same Goeritz matrix.

We proceed to define the Goeritz matrix. If $D$ is a link diagram then let $U$ be the \emph{universe graph} of $D$, i.e., the 4-regular plane graph obtained by restoring the short
segments that were removed to indicate the undercrossing strands. The connected components of the complement of $U$ may be colored checkerboard fashion, so that each edge of $U$ has a shaded region on one side and an unshaded region on the other side. Given a choice $\sigma$ of one of the two checkerboard shadings of the complementary regions of $U$, each crossing of $D$ is assigned a \emph{Goeritz index} as in Figure \ref{figone}: $\eta =1$ if an observer standing in an unshaded region and facing the crossing sees the overpassing arc on her or his right, and $\eta =-1$ if an observer standing in an unshaded region and facing the crossing sees the overpassing arc on her or his left.

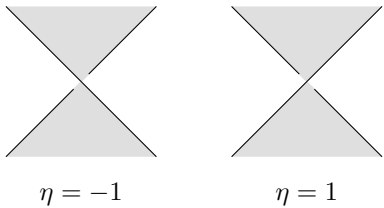
\begin{figure}
\centering
\begin{tikzpicture} 
\draw [thick] [white, fill=lightgray!50] (-2.5,1) -- (-1.5,0) -- (-.5,1);
\draw [thick] [white, fill=lightgray!50] (-2.5,-1) -- (-1.5,0) -- (-1/2,-1);
\draw [thick] [white, fill=lightgray!50] (2.5,1) -- (3/2,0) -- (.5,1);
\draw [thick] [white, fill=lightgray!50] (5/2,-1) -- (3/2,0) -- (1/2,-1);
\draw [thick] (-2.5,1) -- (-1/2,-1);
\draw [thick] (-2.5,-1) -- (-1.6,-.1);
\draw [thick] (-1.4,.1) -- (-1/2,1);
\draw [thick] (2.5,1) -- (1/2,-1);
\draw [thick] (2.5,-1) -- (1.6,-.1);
\draw [thick] (1.4,.1) -- (1/2,1);
\node at (-3/2,-3/2) {$\eta=-1$};
\node at (3/2,-3/2) {$\eta=1$};
\end{tikzpicture}
\caption{The Goeritz index of a crossing.}
\label{figone}
\end{figure}

\begin{definition}
\label{goeritzmat}Let $D$ be a link diagram, and $\sigma$ either of the two
checkerboard shadings of the complementary regions of $D$. Let $R_{1},\dots ,R_{n}$ be the
unshaded complementary regions according to $\sigma$, and for $i,j\in \{1,\dots
,n\}$ let $C_{ij}$ be the set of crossings of $D$ incident on $R_{i}$ and $R_{j}$. Then the (unreduced) \emph{Goeritz matrix} of $D$ with respect to $\sigma$ is the $n\times n$ matrix $G(D,\sigma)$ with these entries: 
\[
G(D,\sigma)_{ij}=%
\begin{cases}
-\sum\limits_{c\in C_{ij}}\eta(c)\text{,} & \text{if }i\neq j\\
-\sum\limits_{k\neq i}G(D,\sigma)_{ik}\text{,} & \text{if }i=j
\end{cases}
\]
\end{definition}

In order to avoid complicating our discussion with trivialities involving the permuting of rows and columns, we formally view a Goeritz matrix as a function $\mathcal{R}_u(D,\sigma) \times \mathcal{R}_u(D,\sigma) \to \mathbb{Z}$, where $\mathcal{R}_u(D,\sigma)$ is the set of unshaded complementary regions according to $\sigma$. That is, we formally view $G(D,\sigma)$ as a matrix whose rows and columns are not ordered. Of course the rows and columns of $G(D,\sigma)$ must be ordered when we display an example of a Goeritz matrix; but the order is not a significant part of the theory.

It is conventional to remove the $n^{th}$ row and column from $G(D,\sigma)$, to obtain a \emph{reduced} Goeritz matrix. As $G(D,\sigma)$ is a symmetric matrix whose rows and columns sum to $0$, one does not gain or lose any information by removing or retaining the $n^{th}$ row and column. We prefer to work with unreduced matrices for the same reason we prefer to formally view the rows and columns of a Goeritz matrix as unordered: we avoid complicating our notation with inconsequential technicalities regarding the arbitrary choice of an ordering of the complementary regions.

It is also conventional to define $G(D,\sigma)$ using a specified shading of $D$, either the shading $\sigma_s$ according to which the unbounded region is shaded, or the shading $\sigma_u$ according to which the unbounded region is unshaded. Again, we prefer not to do this here; and again, our preference is motivated simply by the desire to avoid inconsequential technicalities. In Section 2 we prove Proposition \ref{noshade}, which asserts that even though the choice of a preferred shading convention will certainly alter the Goeritz matrix (or matrices) associated with an individual link diagram, the choice will not alter the set of Goeritz matrices associated with an entire link type. As link types are the focus of knot theory -- not individual link diagrams -- Proposition \ref{noshade} tells us that the choice of a shading convention is inconsequential.

\begin{proposition}
\label{noshade}
Given a link type $[L]$, let
\begin{eqnarray*}
&&\mathcal{G}([L])=\{G(D,\sigma) \mid D\text{ is a diagram of a link of type }[L] \}\text{,}
\\
&&\mathcal{G}_s([L])=\{G(D,\sigma_s) \mid D\text{ is a diagram of a link of type }[L] \}\text{, and }
\\
&&\mathcal{G}_u([L])=\{G(D,\sigma_u) \mid D\text{ is a diagram of a link of type }[L] \}\text{.}
\end{eqnarray*}
Then $\mathcal{G}([L])=\mathcal{G}_s([L]) = \mathcal{G}_u([L])$.
\end{proposition}

At the time the Goeritz matrix was introduced, knot theorists paid little attention to links of more than one component. (This historical focus on one-component links is reflected in the fact that researchers who study links are called ``knot theorists.'') In particular, it is clear that Goeritz did not think about links of more than one component while writing \cite{G}. If he had thought about links, he might have realized that his matrix ignores a simple operation: if we change a link diagram by inserting a small, crossing-free circle into an unshaded region then we have certainly changed the link type, but the corresponding Goeritz matrix is not affected. To allow for this simple operation, we define an adjusted Goeritz matrix in Definition \ref{goeritzmat2} below. This adjustment requires another well-known idea.

\begin{definition}
\label{betadef}Let $D$ be a link diagram, and $\sigma$ either of the two checkerboard shadings of the complementary regions of $D$. Let $\Gamma_s(D,\sigma)$ denote the \emph{shaded checkerboard graph} of $D$. That is, $\Gamma_s(D,\sigma)$ is a graph with a vertex corresponding to each shaded region of $D$ and an edge corresponding to each crossing of $D$, with the edge corresponding to a crossing incident on the vertex or vertices corresponding to the shaded region(s) incident at that crossing. The number of connected components of $\Gamma_s(D,\sigma)$ is denoted $\beta_s(D,\sigma)$.
\end{definition}

The notations $\Gamma_u(D,\sigma),\beta_u(D,\sigma)$ are defined in the same way, using the complementary regions of $D$ that are unshaded according to $\sigma$. (Equivalently, if $\sigma'$ is the other shading of $D$ then $\Gamma_u(D,\sigma)=\Gamma_s(D,\sigma')$ and $\beta_u(D,\sigma)=\beta_s(D,\sigma')$.) The checkerboard graphs of $D$ are often called \emph{Tait graphs} in the literature. 

\begin{definition}
\label{goeritzmat2}Let $\sigma$ be one of the two checkerboard shadings of the complementary regions of a link diagram $D$, and let $B(D,\sigma)$ be the $(\beta_s(D,\sigma)-1) \times (\beta_s(D,\sigma)-1)$ matrix whose entries are all $0$. Then the \emph{adjusted} Goeritz matrix of $D$ with respect to $\sigma$ is 
\[
G^{adj}(D,\sigma)=
\begin{pmatrix}
G(D,\sigma) & 0 \\ 
0 & B(D,\sigma)
\end{pmatrix}.
\]
\end{definition}

Notice that $G^{adj}(D,\sigma)=G(D,\sigma)$ if and only if $B(D,\sigma)$ is empty, i.e., if and only if $\beta_s(D,\sigma)=1$. In particular, $G^{adj}(D,\sigma)=G(D,\sigma)$ if $D$ is a knot diagram. Notice also that adjusted Goeritz matrices succeed in detecting the simple link operation mentioned above: if $D'$ is obtained from $D$ by inserting a small, crossing-free circle into $D$, then $G^{adj}(D',\sigma)$ is obtained from $G^{adj}(D,\sigma)$ by adjoining a row and column of zeroes. $G^{adj}(D',\sigma)$ is the same whether the new component is inserted into a shaded region or an unshaded region.

As before, our preference for not adopting a fixed shading convention is inconsequential, in that the set of adjusted Goeritz matrices associated with a link type would not be affected if we were to choose one convention or the other.

\begin{proposition}
\label{noshadeadj}
Given a link type $[L]$, let
\begin{align*}
&\mathcal{G}^{adj}([L])=\{G^{adj}(D,\sigma) \mid D\text{ is a diagram of a link of type }[L] \}\text{,}
\\
&\mathcal{G}^{adj}_s([L])=\{G^{adj}(D,\sigma_s) \mid D\text{ is a diagram of a link of type }[L] \}\text{, and }
\\
&\mathcal{G}^{adj}_u([L])=\{G^{adj}(D,\sigma_u) \mid D\text{ is a diagram of a link of type }[L] \}\text{.}
\end{align*}
Then $\mathcal{G}^{adj}([L])=\mathcal{G}^{adj}_s([L]) = \mathcal{G}^{adj}_u([L])$.
\end{proposition}

The second equivalence relation we discuss involves \emph{mutation}, part of the theory of tangles introduced by Conway~\cite{Con}. (We refer the reader to~\cite{M} for a detailed presentation of this important notion.) An elementary mutation is performed on a link diagram by first finding a tangle, i.e., a portion of the diagram that is attached to the rest of the diagram by only four arcs, as indicated on the left in Figure \ref{figtwo}; cutting the tangle from the rest of the diagram; rotating the tangle through an angle of $\pi$ in one of the three ways indicated in Figure \ref{figtwo}; and then attaching the rotated tangle in place of the original. (N.b. Only the third kind of elementary mutation involves a rotation in the plane of the diagram.) A mutation is not an ambient isotopy, in general, because the tangle is cut off from the rest of the diagram before rotating it.

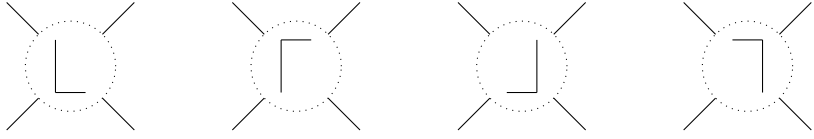
\begin{figure} [bht]
\centering
\begin{tikzpicture}
\draw [thick] [dotted] (-4.5,0) circle (.6 cm);
\draw [thick] (-4.5+.424264,.424264) -- (-4.075736+.424264,.424264+.424264);
\draw [thick] (-4.5-.424264,.424264) -- (-4.5-.424264-.424264,.424264+.424264);
\draw [thick] (-4.5-.424264,-.424264) -- (-4.5-.424264-.424264,-.424264-.424264);
\draw [thick] (-4.5+.424264,-.424264) -- (-4.5+.424264+.424264,-.424264-.424264);
\draw [thick] (-4.7,-.35) -- (-4.7,.35);
\draw [thick] (-4.7,-.35) -- (-4.3,-.35);
\draw [thick] [dotted] (-1.5,0) circle (.6 cm);
\draw [thick] (-1.5+.424264,.424264) -- (-1.075736+.424264,.424264+.424264);
\draw [thick] (-1.5-.424264,.424264) -- (-1.5-.424264-.424264,.424264+.424264);
\draw [thick] (-1.5-.424264,-.424264) -- (-1.5-.424264-.424264,-.424264-.424264);
\draw [thick] (-1.5+.424264,-.424264) -- (-1.5+.424264+.424264,-.424264-.424264);
\draw [thick](-1.7,-.35) -- (-1.7,.35);
\draw [thick] (-1.7,.35) -- (-1.3,.35);
\draw  [thick] [dotted] (1.5,0) circle (.6 cm);
\draw [thick] (6-4.5+.424264,.424264) -- (6-4.075736+.424264,.424264+.424264);
\draw [thick] (6-4.5-.424264,.424264) -- (6-4.5-.424264-.424264,.424264+.424264);
\draw [thick] (6-4.5-.424264,-.424264) -- (6-4.5-.424264-.424264,-.424264-.424264);
\draw [thick] (6-4.5+.424264,-.424264) -- (6-4.5+.424264+.424264,-.424264-.424264);
\draw [thick] (3.4-1.7,-.35) -- (3.4-1.7,.35);
\draw [thick] (3-1.7,-.35) -- (3-1.3,-.35);
\draw [thick] [dotted] (4.5,0) circle (.6 cm);
\draw [thick] (9-4.5+.424264,.424264) -- (9-4.075736+.424264,.424264+.424264);
\draw [thick] (9-4.5-.424264,.424264) -- (9-4.5-.424264-.424264,.424264+.424264);
\draw [thick] (9-4.5-.424264,-.424264) -- (9-4.5-.424264-.424264,-.424264-.424264);
\draw [thick] (9-4.5+.424264,-.424264) -- (9-4.5+.424264+.424264,-.424264-.424264);
\draw [thick] (6.4-1.7,-.35) -- (6.4-1.7,.35);
\draw [thick] (6-1.7,.35) -- (6-1.3,.35);
\end{tikzpicture}
\caption{An elementary mutation of the first, second or third kind replaces the first tangle with the second, third or fourth tangle (respectively).}
\label{figtwo}
\end{figure}

We use adjusted Goeritz matrices and elementary mutations to define two equivalence relations on links, in a standard way.

\begin{definition}
\label{mutant}Two links $L$ and $L^{\prime }$ are \emph{Conway equivalent}, denoted $L\sim_{C} L'$, if there is a finite sequence $L=L_{1},\dots ,L_{k}=L^{\prime }$ of links such that for each $i\in \{2,\dots ,k\}$, at least one of these conditions holds:
\begin{itemize} 
\item $[L_{i-1}]=[L_{i}]$, i.e., $L_{i-1}$ is ambient isotopic to $L_{i}$.
\item $L_{i-1}$ and $L_{i}$ have diagrams $D_{i-1}$ and $D_{i}$ such that $D_{i-1}$ is related to $D_{i}$ by an elementary mutation.
\end{itemize}
\end{definition}

\begin{definition}
\label{goeritzrel}
Two links $L$ and $L^{\prime }$ are \emph{Goeritz equivalent}, denoted $L\sim_{G} L'$, if there is a finite sequence $L=L_{1},\dots ,L_{k}=L^{\prime }$
of links such that for each $i\in \{2,\dots ,k\}$, at least one of these conditions holds:
\begin{itemize} 
\item $[L_{i-1}]=[L_{i}]$, i.e., $L_{i-1}$ is ambient isotopic to $L_{i}$.
\item $L_{i-1}$ and $L_{i}$ have diagrams $D_{i-1}$ and $D_{i}$ with shadings $\sigma_{i-1}$ and $\sigma_{i}$ such that $G^{adj}(D_{i-1},\sigma_{i-1})=G^{adj}(D_{i},\sigma_{i})$.
\end{itemize}
\end{definition}

As a slight abuse of terminology, we will also say that two link diagrams $D$ and $D'$ are Conway equivalent (or Goeritz equivalent) if they represent links $L$ and $L'$ which satisfy Definition \ref{mutant} (or Definition \ref{goeritzrel}). 

The fact that elementary mutations are defined without referring to checkerboard shadings suggests that mutations are not connected with Goeritz matrices or Tait graphs. In Sections 2 -- 7 we prove the opposite:

\begin{theorem}
\label{main}
Conway equivalence and Goeritz equivalence are the same.
\end{theorem}

Considering the importance of invariants in knot theory, it is worth taking a moment to state an obvious consequence of Theorem~\ref{main}:

\begin{corollary}
\label{maincor}
Let $\mathcal{I}$ be an invariant of classical, unoriented link types. Then $\mathcal{I}$ is invariant under mutation if and only if for every choice of a shading $\sigma$ of a diagram $D$ of a link $L$, $\mathcal{I}([L])$ is determined in some way by the adjusted Goeritz matrix $G^{adj}(D,\sigma)$.
\end{corollary}

In Section \ref{sec:or} we adapt the ideas discussed above to oriented links. First we define an oriented version of the Goeritz matrix, $G_{or}(D,\sigma)$, which incorporates the writhe along with the Goeritz index. Then we adapt the arguments of Sections 2 -- 7 to prove oriented versions of Theorem~\ref{main} and Corollary~\ref{maincor}. 

In Sections \ref{sec:cor2} and \ref{sec:cforms} we mention examples of knots that illustrate two important properties of Goeritz matrices. 1. It is possible for two distinct knot types to share some, but not all, of their (adjusted) Goeritz matrices. It follows that in general, a single (adjusted) Goeritz matrix of a link $L$ does not provide enough information to determine all the (adjusted) Goeritz matrices of type $[L]$. This property contrasts with the fact that a single diagram of $L$ does provide enough information to determine all the diagrams of type $[L]$, through Reidemeister moves. 2. It is possible for the Goeritz matrices of two knots to represent equivalent quadratic forms, even if the two knots are not mutants and hence cannot share a Goeritz matrix. This property indicates that even though some Reidemeister moves have the effect of replacing a Goeritz matrix with a conjugate matrix (a matrix representing the same quadratic form), an (adjusted) Goeritz matrix of a link provides strictly more information than the quadratic form the matrix represents.

In the last section of the paper we mention some questions suggested by our results.



Before proceeding we should mention that the ideas presented here were inspired by two earlier theorems. The first is Lipson's theorem that for connected link diagrams, a special form of the Kauffman polynomial is determined by the ordinary (unadjusted) Goeritz matrix~\cite{Lip}. The second is a theorem of Greene~\cite{Gre}, which implies that reduced, alternating link diagrams are characterized up to mutation by their ordinary Goeritz matrices. Our results extend these earlier theorems to arbitrary link diagrams, arbitrary link invariants, and oriented links.

We are grateful to D. A. Smith and W. Watkins, for many hours spent discussing Laplacian matrices of signed graphs and their connections with duality in the plane, links, and quadratic forms. We are also grateful to D. S. Silver and S. G. Williams, for an enlightening correspondence regarding the Tait graphs of link diagrams.

\section{(Not) choosing a shading}
\label{sec:shade}

Recall that if $D$ is a link diagram, $\sigma_{s}$ (resp. $\sigma_{u}$) denotes the shading of complementary regions in which the unbounded region is shaded (resp. unshaded). Definition \ref{goeritzmat} provides two Goeritz matrices, $G(D,\sigma_{s})$ and $G(D,\sigma_{u})$, and Definition \ref{goeritzmat2} provides two adjusted Goeritz matrices, $G^{adj}(D,\sigma_{s})$ and $G^{adj}(D,\sigma_{u})$. Notice that for every crossing of $D$, the Goeritz index is $+1$ with respect to one of the shadings, and $-1$ with respect to the other.

As mentioned in the introduction, it is traditional to adopt a definite shading convention for Goeritz matrices. Propositions \ref{noshade} and \ref{noshadeadj} assert that this tradition is not important when working with Goeritz equivalence. To verify these propositions, consider the following. 

\textbf{Construction}: Given a link diagram $D$, let $D'$ be a diagram obtained by taking a short segment of an arc of $D$ incident on the unbounded complementary region, elongating it into a curve that encloses a large empty portion of the unbounded region, and then swinging the elongated curve under the diagram and back up into the plane of $D$, so that the elongated curve now encloses the rest of $D$. (See Figure~\ref{figthree} for an example.)

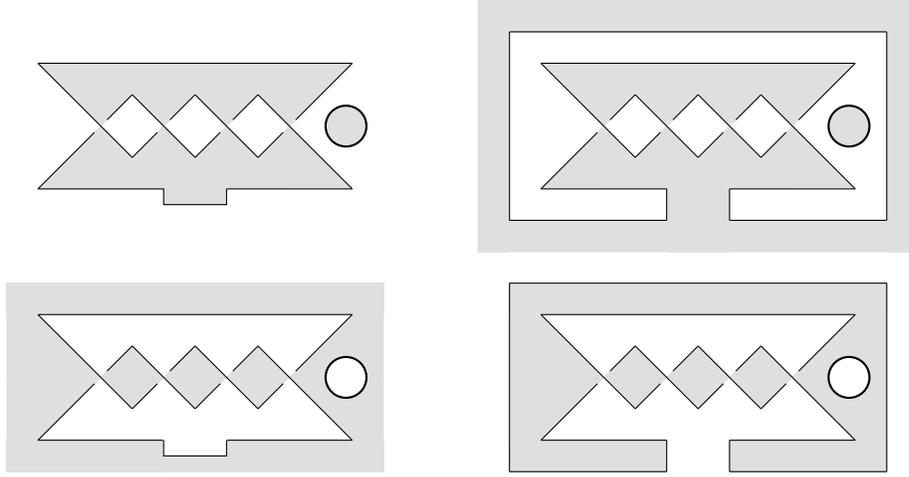
\begin{figure} 
\centering
\begin{tikzpicture} [scale=.0418]
\draw [lightgray!50, fill=lightgray!50] (-130,100) -- (-110,80) -- (-90,100);
\draw [lightgray!50, fill=lightgray!50] (-110,100) -- (-90,80) -- (-70,100);
\draw [lightgray!50, fill=lightgray!50] (-90,100) -- (-70,80) -- (-50,100);
\draw [lightgray!50, fill=lightgray!50] (-70,100) -- (-50,80) -- (-30,100);
\draw [lightgray!50, fill=lightgray!50] (-130,60) -- (-110,80) -- (-90,60);
\draw [lightgray!50, fill=lightgray!50] (-110,60) -- (-90,80) -- (-70,60);
\draw [lightgray!50, fill=lightgray!50] (-90,60) -- (-70,80) -- (-50,60);
\draw [lightgray!50, fill=lightgray!50] (-70,60) -- (-50,80) -- (-30,60);
\draw [lightgray!50, fill=lightgray!50] (30,100) -- (50,80) -- (70,100);
\draw [lightgray!50, fill=lightgray!50] (50,100) -- (70,80) -- (90,100);
\draw [lightgray!50, fill=lightgray!50] (70,100) -- (90,80) -- (110,100);
\draw [lightgray!50, fill=lightgray!50] (90,100) -- (110,80) -- (130,100);
\draw [lightgray!50, fill=lightgray!50] (30,60) -- (50,80) -- (70,60);
\draw [lightgray!50, fill=lightgray!50] (50,60) -- (70,80) -- (90,60);
\draw [lightgray!50, fill=lightgray!50] (70,60) -- (90,80) -- (110,60);
\draw [lightgray!50, fill=lightgray!50] (90,60) -- (110,80) -- (130,60);
\draw [lightgray!50, fill=lightgray!50] (70,60) -- (70,40) -- (90,40) -- (90,60);
\draw [lightgray!50, fill=lightgray!50] (10,40) -- (150,40) -- (150,50) -- (10,50);
\draw [lightgray!50, fill=lightgray!50] (10,120) -- (150,120) -- (150,110) -- (10,110);
\draw [lightgray!50, fill=lightgray!50] (140,120) -- (150,120) -- (150,40) -- (140,40);
\draw [lightgray!50, fill=lightgray!50] (20,120) -- (10,120) -- (10,40) -- (20,40);
\draw [lightgray!50, fill=lightgray!50] (60,-10) -- (70,0) -- (60,10) -- (50,0);
\draw [lightgray!50, fill=lightgray!50] (80,-10) -- (90,0) -- (80,10) -- (70,0);
\draw [lightgray!50, fill=lightgray!50] (100,-10) -- (110,0) -- (100,10) -- (90,0);
\draw [lightgray!50, fill=lightgray!50] (140,20) -- (140,30) -- (20,30) -- (20,20);
\draw [lightgray!50, fill=lightgray!50] (20,30) -- (50,0) -- (20,0);
\draw [lightgray!50, fill=lightgray!50] (20,0) -- (50,0) -- (20,-30);
\draw [lightgray!50, fill=lightgray!50] (140,30) -- (110,0) -- (140,0);
\draw [lightgray!50, fill=lightgray!50] (140,-30) -- (110,0) -- (140,0);
\draw [lightgray!50, fill=lightgray!50] (140,-20) -- (140,-30) -- (90,-30) -- (90,-20);
\draw [lightgray!50, fill=lightgray!50] (70,-20) -- (70,-30) -- (20,-30) -- (20,-20);
\draw [lightgray!50, fill=lightgray!50] (-100,-10) -- (-90,0) -- (-100,10) -- (-110,0);
\draw [lightgray!50, fill=lightgray!50] (-80,-10) -- (-70,0) -- (-80,10) -- (-90,0);
\draw [lightgray!50, fill=lightgray!50] (-60,-10) -- (-50,0) -- (-60,10) -- (-70,0);
\draw [lightgray!50, fill=lightgray!50] (-20,20) -- (-20,30) -- (-140,30) -- (-140,20);
\draw [lightgray!50, fill=lightgray!50] (-140,30) -- (-110,0) -- (-140,0);
\draw [lightgray!50, fill=lightgray!50] (-140,0) -- (-110,0) -- (-140,-30);
\draw [lightgray!50, fill=lightgray!50] (-20,30) -- (-50,0) -- (-20,0);
\draw [lightgray!50, fill=lightgray!50] (-20,-30) -- (-50,0) -- (-20,0);
\draw [lightgray!50, fill=lightgray!50] (-20,-20) -- (-20,-30) -- (-140,-30) -- (-140,-20);
\draw [lightgray!50, fill=lightgray!50] (-90,-20) -- (-90,-30) -- (-140,-30) -- (-140,-20);
\draw [thick] (20,-30) -- (70,-30);
\draw [thick] (90,-30) -- (140,-30);
\draw [thick] (70,-30) -- (70,-20);
\draw [thick] (90,-30) -- (90,-20);
\draw [thick] (30,-20) -- (70,-20);
\draw [thick] (90,-20) -- (130,-20);
\draw [thick] (20,30) -- (140,30);
\draw [thick] (140,-30) -- (140,30);
\draw [thick] (20,-30) -- (20,30);
\draw [thick] (30,20) -- (130,20);
\draw [thick] (30,-20) -- (48,-2);
\draw [thick] (52,2) -- (60,10);
\draw [thick] (60,-10) -- (68,-2);
\draw [thick] (72,2) -- (80,10);
\draw [thick] (80,-10) -- (88,-2);
\draw [thick] (92,2) -- (100,10);
\draw [thick] (100,-10) -- (108,-2);
\draw [thick] (112,2) -- (130,20);
\draw [thick] (30,20) -- (60,-10);
\draw [thick] (60,10) -- (80,-10);
\draw [thick] (80,10) -- (100,-10);
\draw [thick] (100,10) -- (130,-20);
\draw [thick] (20,50) -- (70,50);
\draw [thick] (90,50) -- (140,50);
\draw [thick] (70,50) -- (70,60);
\draw [thick] (90,50) -- (90,60);
\draw [thick] (30,60) -- (70,60);
\draw [thick] (90,60) -- (130,60);
\draw [thick] (20,110) -- (140,110);
\draw [thick] (140,50) -- (140,110);
\draw [thick] (20,50) -- (20,110);
\draw [thick] (30,100) -- (130,100);
\draw [thick] (30,60) -- (48,78);
\draw [thick] (52,82) -- (60,90);
\draw [thick] (60,70) -- (68,78);
\draw [thick] (72,82) -- (80,90);
\draw [thick] (80,70) -- (88,78);
\draw [thick] (92,82) -- (100,90);
\draw [thick] (100,70) -- (108,78);
\draw [thick] (112,82) -- (130,100);
\draw [thick] (30,100) -- (60,70);
\draw [thick] (60,90) -- (80,70);
\draw [thick] (80,90) -- (100,70);
\draw [thick] (100,90) -- (130,60);
\draw [thick] (-130,-20) -- (-90,-20);
\draw [thick] (-70,-20) -- (-30,-20);
\draw [white, fill=white] (-70,-15) -- (-70,-25) -- (-90,-25) -- (-90,-15);
\draw [thick] (-70,-20) -- (-70,-25) -- (-90,-25) -- (-90,-20);
\draw [thick] (-130,20) -- (-30,20);
\draw [thick] (-130,-20) -- (-112,-2);
\draw [thick] (-108,2) -- (-100,10);
\draw [thick] (-100,-10) -- (-92,-2);
\draw [thick] (-88,2) -- (-80,10);
\draw [thick] (-80,-10) -- (-72,-2);
\draw [thick] (-68,2) -- (-60,10);
\draw [thick] (-60,-10) -- (-52,-2);
\draw [thick] (-48,2) -- (-30,20);
\draw [thick] (-130,20) -- (-100,-10);
\draw [thick] (-100,10) -- (-80,-10);
\draw [thick] (-80,10) -- (-60,-10);
\draw [thick] (-60,10) -- (-30,-20);
\draw [lightgray!50, fill=lightgray!50] (-130,60) -- (-90,60) -- (-90,55) -- (-70,55) -- (-70,60);
\draw [thick] (-130,60) -- (-90,60) -- (-90,55) -- (-70,55) -- (-70,60);
\draw [thick] (-30,60) -- (-70,60);
\draw [thick] (-130,100) -- (-30,100);
\draw [thick] (-130,60) -- (-112,78);
\draw [thick] (-108,82) -- (-100,90);
\draw [thick] (-100,70) -- (-92,78);
\draw [thick] (-88,82) -- (-80,90);
\draw [thick] (-80,70) -- (-72,78);
\draw [thick] (-68,82) -- (-60,90);
\draw [thick] (-60,70) -- (-52,78);
\draw [thick] (-48,82) -- (-30,100);
\draw [thick] (-130,100) -- (-100,70);
\draw [thick] (-100,90) -- (-80,70);
\draw [thick] (-80,90) -- (-60,70);
\draw [thick] (-60,90) -- (-30,60);
\draw [thick, fill=lightgray!50] (-32,80) circle (6.5 cm);
\draw [thick, fill=white] (-32,0) circle (6.5 cm);
\draw [thick, fill=lightgray!50] (128,80) circle (6.5 cm);
\draw [thick, fill=white] (128,0) circle (6.5 cm);
\end{tikzpicture}
\caption{The two shadings of $D$ on the left have the same (adjusted) Goeritz matrices as the two shadings of $D'$ on the right.}
\label{figthree}
\end{figure}

It is easy to see that the construction yields a diagram $D^{\prime }$ of the same link type as $D$, whose crossings and regions correspond directly to those of $D$. In particular, the unbounded region of $D'$ corresponds to the bounded region of $D$ incident on the short segment used by the construction. It follows that the two shadings $\sigma_{s}^{\prime },\sigma_{u}^{\prime }$ of $D'$ have $G(D,\sigma_{s})=G(D^{\prime},\sigma_{u}^{\prime })$, $G(D,\sigma_{u})=G(D^{\prime },\sigma_{s}^{\prime })$, $G^{adj}(D,\sigma_{s})=G^{adj}(D^{\prime},\sigma_{u}^{\prime })$ and $G^{adj}(D,\sigma_{u})=G^{adj}(D^{\prime },\sigma_{s}^{\prime })$. These equalities verify Propositions \ref{noshade} and \ref{noshadeadj}.

The following inductive generalization of Propositions \ref{noshade} and \ref{noshadeadj} is proven by iterating the construction.

\begin{proposition}
\label{infpt}
Let $R$ be any one of the complementary regions of a link diagram $D$. Then there are a link diagram $D'$ of the same link type as $D$ and a pair of bijections $f:\{ $complementary regions of $D\} \to \{$complementary regions of $D' \}$ and $g:\{ $crossings of $D\} \to \{$crossings of $D' \}$, with the following properties:
\begin{enumerate}
\item The unbounded complementary region of $D'$ is $f(R)$.
\item Let $\sigma_1,\sigma_2$ be the two shadings of $D$, with $R$ shaded in $\sigma_1$. Let $\sigma'_s,\sigma'_u$ be the two shadings of $D'$, with the unbounded region shaded in $\sigma'_s$. Then for each crossing $c$ of $D$, the Goeritz signs of $c$ with respect to $\sigma_1,\sigma_2$ are the same as the Goeritz signs of $g(c)$ with respect to $\sigma'_s,\sigma'_u$ (respectively).

\item If the regions incident on a crossing $c$ of $D$ are $R_1,R_2,R_3$ and $R_4$ then the regions incident on $g(c)$ are $f(R_1),f(R_2),f(R_3)$ and $f(R_4)$.

\item The bijections $f,g$ define graph isomorphisms $\Gamma_u(D,\sigma_1) \cong \Gamma_u(D',\sigma'_s)$ and $\Gamma_u(D,\sigma_2) \cong \Gamma_u(D',\sigma'_u)$.

\item The bijection $f$ defines matrix equalities $G(D,\sigma_{1})=G(D^{\prime},\sigma_{s}^{\prime })$, $G(D,\sigma_{2})=G(D^{\prime },\sigma_{u}^{\prime })$, $G^{adj}(D,\sigma_{1})=G^{adj}(D^{\prime},\sigma_{s}^{\prime })$ and $G^{adj}(D,\sigma_{2})=G^{adj}(D^{\prime },\sigma_{u}^{\prime })$.
\end{enumerate}
\end{proposition}

See Figure \ref{figfour} for an example obtained by using the construction twice. The bijection $f$ is indicated by the region labels 1 -- 7.

\begin{figure} [bht]
\centering
\begin{tikzpicture} [scale=.043]
\draw [thick] (-130,-20) -- (-30,-20);
\draw [thick] (-130,20) -- (-30,20);
\draw [thick] (-60,10) -- (-52,2);
\draw [thick] (-30,-20) -- (-48,-2);
\draw [thick] (-130,-20) -- (-112,-2);
\draw [thick] (-108,2) -- (-100,10);
\draw [thick] (-100,-10) -- (-92,-2);
\draw [thick] (-88,2) -- (-80,10);
\draw [thick] (-80,10) -- (-72,2);
\draw [thick] (-68,-2) -- (-60,-10);
\draw [thick] (-60,-10) -- (-30,20);
\draw [thick] (-130,20) -- (-100,-10);
\draw [thick] (-100,10) -- (-80,-10);
\draw [thick] (-60,10) -- (-80,-10);
\draw [thick] (160-130,-20) -- (160-55,-20);
\draw [white, fill=white] (160-44,-19) -- (160-56,-19) -- (160-44,-31) -- (160-34,-31);
\draw [white, fill=white] (160-31,-19) -- (160-41,-19) -- (160-29,-31) -- (160-19,-31);
\draw [thick] (160-10,-40) -- (160-18,-40);
\draw [thick] (160-45,-30) -- (160-55,-20);
\draw [thick] (160-130,20) -- (160-30,20);
\draw [thick] (160-10,-40) -- (160-10,40);
\draw [thick] (160-150,40) -- (160-10,40);
\draw [thick] (160-150,40) -- (160-150,-40);
\draw [thick] (160-26,-40) -- (160-150,-40);
\draw [thick] (160-45,-30) -- (160-140,-30);
\draw [thick] (160-140,30) -- (160-140,-30);
\draw [thick] (160-140,30) -- (160-20,30);
\draw [thick] (160-20,-30) -- (160-20,30);
\draw [thick] (160-20,-30) -- (160-48,-2);
\draw [thick] (160-60,10) -- (160-52,2);
\draw [thick] (160-130,-20) -- (160-112,-2);
\draw [thick] (160-108,2) -- (160-100,10);
\draw [thick] (160-100,-10) -- (160-92,-2);
\draw [thick] (160-88,2) -- (160-80,10);
\draw [thick] (160-68,-2) -- (160-60,-10);
\draw [thick] (160-80,10) -- (160-72,2);
\draw [thick] (160-60,-10) -- (160-58,-8);
\draw [thick] (160-26,-40) -- (160-58,-8);
\draw [thick] (160-18,-40) -- (160-54,-4);
\draw [thick] (160-80,-10) -- (160-60,10);
\draw [thick] (160-54,-4) -- (160-30,20);
\draw [thick] (160-130,20) -- (160-100,-10);
\draw [thick] (160-100,10) -- (160-80,-10);
\draw [thick, fill=white] (-32,0) circle (6.5 cm);
\draw [thick, fill=white] (160-32,0) circle (6.5 cm);
\node at (-125,0) {1};
\node at (-110,12) {2};
\node at (-100,0) {3};
\node at (-80,0) {4};
\node at (-60,0) {6};
\node at (-70,-12) {5};
\node at (-32,0) {7};
\node at (160-125,0) {1};
\node at (160-110,12) {2};
\node at (160-100,0) {3};
\node at (160-80,0) {4};
\node at (160-60,0) {6};
\node at (160-70,-12) {5};
\node at (160-32,0) {7};
\end{tikzpicture}
\caption{$G(D,\sigma_{s})=G(D^{\prime},\sigma_{s}^{\prime })$, $G(D,\sigma_{u})=G(D^{\prime },\sigma_{u}^{\prime })$, $G^{adj}(D,\sigma_{s})=G^{adj}(D^{\prime},\sigma_{s}^{\prime })$ and  $G^{adj}(D,\sigma_{u})=G^{adj}(D^{\prime },\sigma_{u}^{\prime })$. The unbounded regions of $D$ and $D'$ do not correspond to each other, however.}
\label{figfour}
\end{figure}
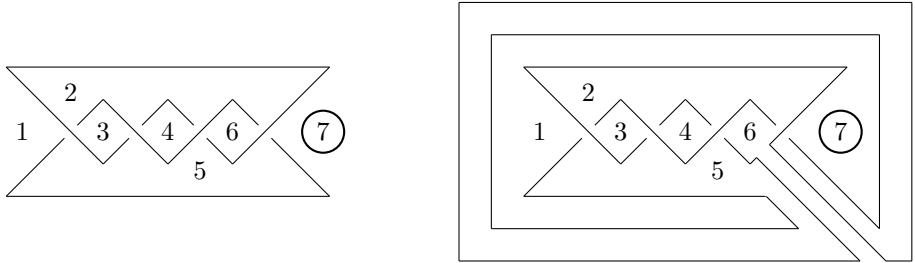

Even though Propositions \ref{noshade} and \ref{noshadeadj} assert that it makes no difference whether or not we adopt a fixed shading convention, we have found that the idea of allowing both shadings of a link diagram makes it much easier to understand the properties of Goeritz matrices. As far as we know, this useful idea is due to Nanyes~\cite{N}.

\section{Conway equivalence $\implies$ Goeritz equivalence}
\label{sec:cg}

In this section we prove the less difficult direction of Theorem \ref{main}. 

It is easy to see that if $D$ and $D'$ are related through the first kind of elementary mutation and $\sigma$ and $\sigma'$ are the shadings indicated in Figure~\ref{figfive}, then the unshaded checkerboard graphs $\Gamma_u(D,\sigma)$ and $\Gamma_u(D',\sigma')$ are isomorphic; they differ only in the way they are drawn in the plane. It is also easy to see that the tangle rotation does not change the Goeritz indices of crossings, so $G(D,\sigma)=G(D',\sigma')$. 
 
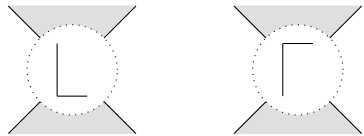
\begin{figure} [bht]
\centering
\begin{tikzpicture}
\draw [lightgray!50, fill=lightgray!50] (-4.075736+.424264,.424264+.424264) -- (-4.5-.424264-.424264,.424264+.424264) -- (-4.5,0);
\draw [lightgray!50, fill=lightgray!50] (-4.075736+.424264,-.424264-.424264) -- (-4.5-.424264-.424264,-.424264-.424264) -- (-4.5,0);
\draw [lightgray!50, fill=lightgray!50] (-1.075736+.424264,.424264+.424264) -- (-1.5-.424264-.424264,.424264+.424264) -- (-1.5,0);
\draw [lightgray!50, fill=lightgray!50] (-1.075736+.424264,-.424264-.424264) -- (-1.5-.424264-.424264,-.424264-.424264) -- (-1.5,0);
\draw [thick] [dotted] [fill=white] (-4.5,0) circle (.6 cm);
\draw [thick] (-4.5+.424264,.424264) -- (-4.075736+.424264,.424264+.424264);
\draw [thick] (-4.5-.424264,.424264) -- (-4.5-.424264-.424264,.424264+.424264);
\draw [thick] (-4.5-.424264,-.424264) -- (-4.5-.424264-.424264,-.424264-.424264);
\draw [thick] (-4.5+.424264,-.424264) -- (-4.5+.424264+.424264,-.424264-.424264);
\draw [thick] (-4.7,-.35) -- (-4.7,.35);
\draw [thick] (-4.7,-.35) -- (-4.3,-.35);
\draw [thick] [dotted] (-1.5,0) [fill=white] circle (.6 cm);
\draw [thick] (-1.5+.424264,.424264) -- (-1.075736+.424264,.424264+.424264);
\draw [thick] (-1.5-.424264,.424264) -- (-1.5-.424264-.424264,.424264+.424264);
\draw [thick] (-1.5-.424264,-.424264) -- (-1.5-.424264-.424264,-.424264-.424264);
\draw [thick] (-1.5+.424264,-.424264) -- (-1.5+.424264+.424264,-.424264-.424264);
\draw [thick] (-1.7,-.35) -- (-1.7,.35);
\draw [thick] (-1.7,.35) -- (-1.3,.35);
\end{tikzpicture}
\caption{An elementary mutation of the first kind preserves one $G^{adj}(D,\sigma)$ matrix.}
\label{figfive}
\end{figure}

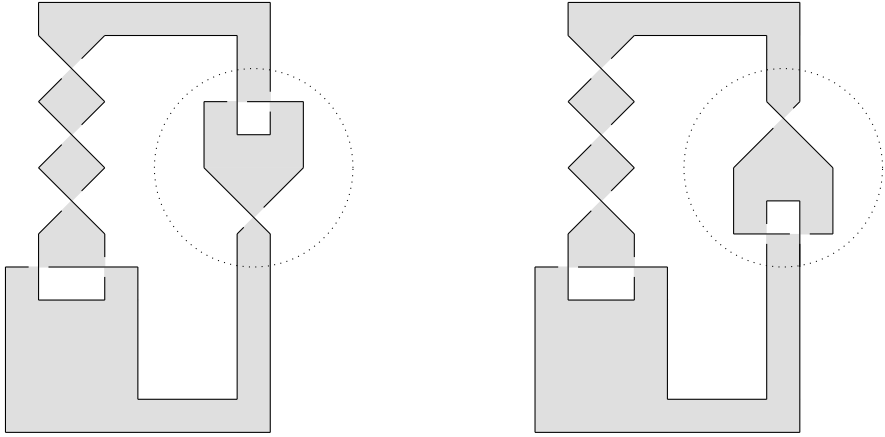
\begin{figure} [bht]
\centering
\begin{tikzpicture} [scale=0.44]
\draw [lightgray!50, fill=lightgray!50] (3,13) -- (3,12) -- (1,12) -- (1,13);
\draw [lightgray!50, fill=lightgray!50] (3,12) -- (1,12) -- (2,11);
\draw [lightgray!50, fill=lightgray!50] (3,10) -- (2,9) -- (1,10) -- (2,11);
\draw [lightgray!50, fill=lightgray!50] (3,8) -- (2,7) -- (1,8) -- (2,9);
\draw [lightgray!50, fill=lightgray!50] (3,6) -- (3,5) -- (1,5) -- (1,6) -- (2,7);
\draw [lightgray!50, fill=lightgray!50] (3,4) -- (3,5) -- (4,5) -- (4,4);
\draw [lightgray!50, fill=lightgray!50] (0,4) -- (0,5) -- (1,5) -- (1,4);
\draw [lightgray!50, fill=lightgray!50] (0,4) -- (4,4) -- (2,2);
\draw [lightgray!50, fill=lightgray!50] (0,0) -- (4,0) -- (4,4) -- (0,4);
\draw [lightgray!50, fill=lightgray!50] (7,6) -- (7.5,6.5) -- (8,6);
\draw [lightgray!50, fill=lightgray!50] (7,6) -- (8,6) -- (8,0) -- (7,0);
\draw [lightgray!50, fill=lightgray!50] (6,8) -- (7.5,6.5) -- (9,8);
\draw [lightgray!50, fill=lightgray!50] (7,5) -- (7,6) -- (8,6) -- (8,5);
\draw [lightgray!50, fill=lightgray!50] (9,9) -- (9,10) -- (8,10) -- (8,9);
\draw [lightgray!50, fill=lightgray!50] (7,9) -- (7,10) -- (6,10) -- (6,9);
\draw [lightgray!50, fill=lightgray!50] (8,0) -- (8,1) -- (3,1) -- (3,0);
\draw [lightgray!50, fill=lightgray!50] (7,13) -- (7,10) -- (8,10) -- (8,13);
\draw [lightgray!50, fill=lightgray!50] (3,13) -- (3,12) -- (8,12) -- (8,13);
\draw [lightgray!50, fill=lightgray!50] (6,8) -- (6,9) -- (9,9) -- (9,8);
\draw [lightgray!50, fill=lightgray!50] (6,8) -- (9,8);
\draw [lightgray!50, fill=lightgray!50] (19,13) -- (19,12) -- (17,12) -- (17,13);
\draw [lightgray!50, fill=lightgray!50] (19,12) -- (17,12) -- (18,11);
\draw [lightgray!50, fill=lightgray!50] (19,10) -- (18,9) -- (17,10) -- (18,11);
\draw [lightgray!50, fill=lightgray!50] (19,8) -- (18,7) -- (17,8) -- (18,9);
\draw [lightgray!50, fill=lightgray!50] (19,6) -- (19,5) -- (17,5) -- (17,6) -- (18,7);
\draw [lightgray!50, fill=lightgray!50] (19,4) -- (19,5) -- (20,5) -- (20,4);
\draw [lightgray!50, fill=lightgray!50] (16,4) -- (16,5) -- (17,5) -- (17,4);
\draw [lightgray!50, fill=lightgray!50] (16,4) -- (20,4) -- (18,2);
\draw [lightgray!50, fill=lightgray!50] (16,0) -- (20,0) -- (20,4) -- (16,4);
\draw [lightgray!50, fill=lightgray!50] (19,0) -- (19,1) -- (24,1) -- (24,0);
\draw [lightgray!50, fill=lightgray!50] (23,6) -- (23,1) -- (24,1) -- (24,6);
\draw [lightgray!50, fill=lightgray!50] (23,6) -- (22,6) -- (22,8) -- (23,8);
\draw [lightgray!50, fill=lightgray!50] (25,6) -- (24,6) -- (24,8) -- (25,8);
\draw [lightgray!50, fill=lightgray!50] (25,7) -- (22,7) -- (22,8) -- (25,8);
\draw [lightgray!50, fill=lightgray!50] (25,8) -- (22,8) -- (23.5,9.5);
\draw [lightgray!50, fill=lightgray!50] (24,13) -- (24,10) -- (23.5,9.5) -- (23,10) -- (23,13);
\draw [lightgray!50, fill=lightgray!50] (24,12) -- (24,13) -- (22,13) -- (22,12);
\draw [lightgray!50, fill=lightgray!50] (24,12) -- (24,13) -- (18,13) -- (18,12);
\draw [thick] [dotted] (7.5,8) circle (3 cm);
\draw [thick] [dotted] (23.5,8) circle (3 cm);
\draw [thick] (0,0) -- (8,0);
\draw [thick] (0,0) -- (0,5);
\draw [thick] (4,1) -- (7,1);
\draw [thick] (4,1) -- (4,5);
\draw [thick] (0,5) -- (.7,5);
\draw [thick] (4,5) -- (1.3,5);
\draw [thick] (3,4.7) -- (3,4);
\draw [thick] (1,4) -- (3,4);
\draw [thick] (1,4) -- (1,6);
\draw [thick] (3,8) -- (2.3,7.3);
\draw [thick] (1,6) -- (1.7,6.7);
\draw [thick] (3,5.3) -- (3,6);
\draw [thick] (1,8) -- (3,6);
\draw [thick] (1,10) -- (3,8);
\draw [thick] (1,8) -- (1.7,8.7);
\draw [thick] (2.3,9.3) -- (3,10);
\draw [thick] (1,10) -- (1.7,10.7);
\draw [thick] (2.3,11.3) -- (3,12);
\draw [thick] (1,12) -- (3,10);
\draw [thick] (1,12) -- (1,13);
\draw [thick] (8,13) -- (1,13);
\draw [thick] (3,12) -- (7,12);
\draw [thick] (7,9) -- (7,12);
\draw [thick] (8,13) -- (8,10.3);
\draw [thick] (8,9) -- (8,9.7);
\draw [thick] (8,9) -- (7,9);
\draw [thick] (7.3,10) -- (9,10);
\draw [thick] (6.7,10) -- (6,10);
\draw [thick] (9,8) -- (9,10);
\draw [thick] (6,8) -- (6,10);
\draw [thick] (6,8) -- (8,6);
\draw [thick] (7,1) -- (7,6);
\draw [thick] (8,0) -- (8,6);
\draw [thick] (7.8,6.8) -- (9,8);
\draw [thick] (7,6) -- (7.2,6.2);
\draw [thick] (20-4,0) -- (28-4,0);
\draw [thick] (20-4,0) -- (20-4,5);
\draw [thick] (24-4,1) -- (27-4,1);
\draw [thick] (24-4,1) -- (24-4,5);
\draw [thick] (20-4,5) -- (20.7-4,5);
\draw [thick] (24-4,5) -- (21.3-4,5);
\draw [thick] (23-4,4.7) -- (23-4,4);
\draw [thick] (21-4,4) -- (23-4,4);
\draw [thick] (21-4,4) -- (21-4,6);
\draw [thick] (23-4,8) -- (22.3-4,7.3);
\draw [thick] (21-4,6) -- (21.7-4,6.7);
\draw [thick] (23-4,5.3) -- (23-4,6);
\draw [thick] (21-4,8) -- (23-4,6);
\draw [thick] (21-4,10) -- (23-4,8);
\draw [thick] (21-4,8) -- (21.7-4,8.7);
\draw [thick] (22.3-4,9.3) -- (23-4,10);
\draw [thick] (21-4,10) -- (21.7-4,10.7);
\draw [thick] (22.3-4,11.3) -- (23-4,12);
\draw [thick] (21-4,12) -- (23-4,10);
\draw [thick] (21-4,12) -- (21-4,13);
\draw [thick] (28-4,13) -- (21-4,13);
\draw [thick] (23-4,12) -- (27-4,12);
\draw [thick] (24,5.7) -- (24,0);
\draw [thick] (24,7) -- (24,5);
\draw [thick] (23,1) -- (23,5.7);
\draw [thick] (23,6.3) -- (23,7);
\draw [thick] (24,7) -- (23,7);
\draw [thick] (24.3,6) -- (25,6);
\draw [thick] (22,6) -- (23.7,6);
\draw [thick] (25,8) -- (25,6);
\draw [thick] (22,8) -- (22,6);
\draw [thick] (25,8) -- (23,10);
\draw [thick] (23,12) -- (23,10);
\draw [thick] (23.8,9.8) -- (24,10);
\draw [thick] (23.2,9.2) -- (22,8);
\draw [thick] (24,13) -- (24,10);
\end{tikzpicture}
\caption{$\Gamma_u(D,\sigma)$ and $\Gamma_u(D',\sigma')$ are isomorphic, while $\Gamma_s(D,\sigma)$ and $\Gamma_s(D',\sigma')$ are related by a Whitney twist.}
\label{figsix}
\end{figure}

On the other hand, in the situation indicated in Figure~\ref{figfive} the shaded checkerboard graphs $\Gamma_s(D,\sigma)$ and $\Gamma_s(D',\sigma')$ differ by a graph operation called a ``Whitney twist'' (or a ``2-switching''). Under this operation the vertices corresponding to the two shaded regions pictured in Figure \ref{figfive} exchange their incidences on edges corresponding to crossings inside the rotated tangle. (If the two shaded regions of Figure \ref{figfive} are portions of a single shaded region of $D$, the edge exchange has no effect.) As indicated by the example in Figure \ref{figsix}, a Whitney twist can certainly yield a nonisomorphic graph. (The shaded checkerboard graph from the right side of Figure \ref{figsix} has a vertex of degree four, but the shaded checkerboard graph from the left side does not have a vertex of degree four.) However, a Whitney twist always preserves the number of connected components; for completeness we provide the easy argument in Lemma~\ref{twist} below. 

\begin{definition}
\label{twistdef}
Let $\Gamma$ be a graph with subgraphs $\Gamma_1$ and $\Gamma_2$ such that $\Gamma = \Gamma_1 \cup \Gamma_2$, $E(\Gamma_1) \cap E(\Gamma_2) = \emptyset$, and $V(\Gamma_1) \cap V(\Gamma_2) = \{v,w\}$. Let $\Gamma'_2$ be the graph obtained from $\Gamma_2$ by exchanging all vertex-edge incidences involving $v$ for vertex-edge incidences involving $w$, and vice versa. Then the graph $\Gamma' = \Gamma_1 \cup \Gamma'_2$ is obtained from $\Gamma$ by a \emph{Whitney twist} with respect to $\{v,w\}$.
\end{definition}

Notice that Definition~\ref{twistdef} allows some fairly trivial cases: we might have $v=w$, $v$ or $w$ might be a cutpoint, or $v$ and $w$ might lie in two connected components of $\Gamma$. Some references do not call all of these cases ``Whitney twists.''

\begin{lemma}
\label{twist}
If $\Gamma'$ is obtained from $\Gamma$ by a Whitney twist, then $\Gamma$ and $\Gamma'$ have the same number of connected components.
\end{lemma}

\begin{proof}
The lemma follows from three obvious properties. A path of $\Gamma$ that does not include $v$ or $w$ is a path of $\Gamma'$, and vice versa. An element of $V(\Gamma_1)$ is connected to $v$ (resp. $w$) by a path in $\Gamma$ if and only if it is connected to $v$ (resp. $w$) by a path in $\Gamma'$. And an element of $V(\Gamma_2)-\{v,w\}$ is connected to $v$ (resp. $w$) by a path in $\Gamma$ if and only if it is connected to $w$ (resp. $v$) by a path in $\Gamma'$.
\end{proof}

Lemma~\ref{twist} completes the verification that $G^{adj}(D,\sigma)=G^{adj}(D',\sigma')$ in the situation of Figure~\ref{figfive}. A similar analysis indicates that the second kind of elementary mutation preserves the Goeritz matrix and adjusted Goeritz matrix associated with the shadings indicated in Figure~\ref{figseven}. 
 
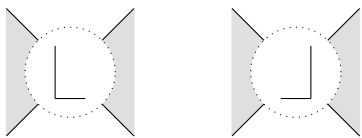
\begin{figure} [bht]
\centering
\begin{tikzpicture}
\draw [lightgray!50, fill=lightgray!50] (4-1.075736+.424264,.424264+.424264) -- (4-1.075736+.424264,-.424264-.424264) -- (2.5,0);
\draw [lightgray!50, fill=lightgray!50] (4-1.5-.424264-.424264,-.424264-.424264) -- (4-1.5-.424264-.424264,.424264+.424264) -- (2.5,0);
\draw [lightgray!50, fill=lightgray!50] (7-1.075736+.424264,.424264+.424264) -- (7-1.075736+.424264,-.424264-.424264) -- (5.5,0);
\draw [lightgray!50, fill=lightgray!50] (7-1.5-.424264-.424264,-.424264-.424264) -- (7-1.5-.424264-.424264,.424264+.424264) -- (5.5,0);
\draw [thick] [dotted] (1+1.5,0) [fill=white] circle (.6 cm);
\draw [thick] (1+6-4.5+.424264,.424264) -- (1+6-4.075736+.424264,.424264+.424264);
\draw [thick] (1+6-4.5-.424264,.424264) -- (1+6-4.5-.424264-.424264,.424264+.424264);
\draw [thick] (1+6-4.5-.424264,-.424264) -- (1+6-4.5-.424264-.424264,-.424264-.424264);
\draw [thick] (1+6-4.5+.424264,-.424264) -- (1+6-4.5+.424264+.424264,-.424264-.424264);
\draw [thick] (1+3-1.7,-.35) -- (1+3-1.7,.35);
\draw [thick] (1+3-1.7,-.35) -- (1+3-1.3,-.35);
\draw [thick] [dotted] (1+4.5,0) [fill=white] circle (.6 cm);
\draw [thick] (1+9-4.5+.424264,.424264) -- (1+9-4.075736+.424264,.424264+.424264);
\draw [thick] (1+9-4.5-.424264,.424264) -- (1+9-4.5-.424264-.424264,.424264+.424264);
\draw [thick] (1+9-4.5-.424264,-.424264) -- (1+9-4.5-.424264-.424264,-.424264-.424264);
\draw [thick] (1+9-4.5+.424264,-.424264) -- (1+9-4.5+.424264+.424264,-.424264-.424264);
\draw [thick] (1+6.4-1.7,-.35) -- (1+6.4-1.7,.35);
\draw [thick] (1+6-1.7,-.35) -- (1+6-1.3,-.35);
\end{tikzpicture}
\caption{An elementary mutation of the second kind also preserves one $G^{adj}(D,\sigma)$ matrix.}
\label{figseven}
\end{figure}

An elementary mutation of the third kind does not preserve either Goeritz matrix, in general. (Both checkerboard graphs are subjected to Whitney twists.) However, an elementary mutation of the third kind can be performed by first performing an elementary mutation of the first kind, and then performing an elementary mutation of the second kind. It follows that every finite sequence $L=L_{1},\dots ,L_{k}=L^{\prime }$ satisfying Definition \ref{mutant} yields a finite sequence $L=L'_{1},\dots,L'_{k^{\prime }}=L^{\prime }$ satisfying Definition \ref{goeritzrel}.

\section{Strongly normal shadings of link diagrams}
\label{sec:normal}

The following definition is a refinement of Goeritz's idea of ``normal diagrams,'' discussed in Section 1.5 of Reidemeister's classic monograph \cite{R}.

\begin{definition} 
\label{strongnormaldiagram}
Let $D$ be a link diagram, with a shading $\sigma$. Let $R_1,\dots,R_n$ be the unshaded complementary regions according to $\sigma$, and for $i, j \in \{1,\dots,n\}$ let $C_{ij}$ be the set of crossings of $D$ that are incident only on $R_i$ and $R_j$. Then $\sigma$ is a \emph{strongly normal shading} of $D$ if these five properties hold:
\begin{enumerate}
\item The unbounded complementary region is shaded, i.e., $\sigma=\sigma_s$.
\item The shaded checkerboard graph $\Gamma_s(D,\sigma)$ is connected, i.e., $\beta_s(D,\sigma)=1$.
\item $C_{ii}=\emptyset$ for every $i$.
\item No $C_{ij}$ contains both a crossing with $\eta=1$ and a crossing with $\eta=-1$.
\item If $\left \vert C_{ij} \right \vert > 1$ then all of the crossings in $C_{ij}$ occur along a single two-stranded braid. (See Figure \ref{figeight}.)
\end{enumerate}
\end{definition}

\begin{figure} [bht]
\centering 
\begin{tikzpicture} [scale=0.5]
\draw [lightgray!50, fill=lightgray!50] (1+2,0) -- (1+3,1) -- (1+2,2);
\draw [lightgray!50, fill=lightgray!50] (1+8,0) -- (1+7,1) -- (1+8,2) -- (1+9,1);
\draw [lightgray!50, fill=lightgray!50] (1+4,0) -- (1+3,1) -- (1+4,2) -- (1+5,1);
\draw [lightgray!50, fill=lightgray!50] (1+6,0) -- (1+5,1) -- (1+6,2) -- (1+7,1);
\draw [lightgray!50, fill=lightgray!50] (1+10,0) -- (1+9,1) -- (1+10,2);
\draw [lightgray!50, fill=lightgray!50] (15,0) -- (16,1) -- (15,2);
\draw [lightgray!50, fill=lightgray!50] (17,0) -- (16,1) -- (17,2) -- (18,1);
\draw [lightgray!50, fill=lightgray!50] (19,0) -- (18,1) -- (19,2) -- (20,1);
\draw [lightgray!50, fill=lightgray!50] (21,0) -- (20,1) -- (21,2) -- (22,1);
\draw [lightgray!50, fill=lightgray!50] (23,0) -- (22,1) -- (23,2);
\draw [thick] (1+2,0) -- (1+4,2);
\draw [thick] (1+4,0) -- (1+6,2);
\draw [thick] (1+6,0) -- (1+8,2);
\draw [thick] (1+8,0) -- (1+10,2);
\draw [thick] (1+2,2) -- (1+2.8,1.2);
\draw [thick] (1+3.2,0.8) -- (1+4,0);
\draw [thick] (1+4,2) -- (1+4.8,1.2);
\draw [thick] (1+5.2,0.8) -- (1+6,0);
\draw [thick] (1+6,2) -- (1+6.8,1.2);
\draw [thick] (1+7.2,0.8) -- (1+8,0);
\draw [thick] (1+8,2) -- (1+8.8,1.2);
\draw [thick] (1+9.2,0.8) -- (1+10,0);
\draw [thick] (15,2) -- (17,0);
\draw [thick] (17,2) -- (19,0);
\draw [thick] (19,2) -- (21,0);
\draw [thick] (21,2) -- (23,0);
\draw [thick] (15,0) -- (15.8,0.8);
\draw [thick] (16.2,1.2) -- (17,2);
\draw [thick] (17,0) -- (17.8,0.8);
\draw [thick] (18.2,1.2) -- (19,2);
\draw [thick] (19,0) -- (19.8,0.8);
\draw [thick] (20.2,1.2) -- (21,2);
\draw [thick] (21,0) -- (21.8,0.8);
\draw [thick] (22.2,1.2) -- (23,2);
\end{tikzpicture}
\caption{Two-stranded braids of four crossings.}
\label{figeight}
\end{figure}
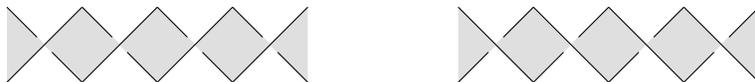

Our next proposition shows that Definition~\ref{strongnormaldiagram} provides standard representatives for link diagrams under the combination of $\sim_C$ and $\sim_G$. The proof combines and extends discussions of Lipson \cite{L} and Reidemeister \cite{R}.
\begin{proposition}
\label{goeritzprop}
If $\sigma$ is a shading of a link diagram $D$ then there is a link diagram $D'$ with a strongly normal shading $\sigma'$, such that $D \sim_C D'$ and $G^{adj}(D,\sigma)=G^{adj}(D',\sigma')$.
\end{proposition}
\begin{proof}
According to Proposition \ref{noshadeadj}, there is a diagram $D'$ of the same link type as $D$, which has $G^{adj}(D,\sigma)=G^{adj}(D',\sigma')$ and $\sigma'=\sigma'_s$.

\begin{figure} [bht]
\centering
\begin{tikzpicture} [scale=0.47]
\draw [lightgray!50, fill=lightgray!50] (2,0+12) -- (11,0+12) -- (11,1+12) -- (2,1+12);
\draw [lightgray!50, fill=lightgray!50] (1,1+12) -- (1,9+12) -- (2,9+12) -- (2,1+12);
\draw [lightgray!50, fill=lightgray!50] (2,9+12) -- (11,9+12) -- (11,10+12) -- (2,10+12);
\draw [lightgray!50, fill=lightgray!50] (11,9+12) -- (11,1+12) -- (12,1+12) -- (12,9+12);
\draw [thick] (11,10+12) -- (11,1.3+12);
\draw [thick] (11,0+12) -- (11,0.7+12);
\draw [thick] (2,10+12) -- (2,9.3+12);
\draw [thick] (2,1.3+12) -- (2,8.7+12);
\draw [thick] (2,0+12) -- (2,0.7+12);
\draw [thick] (1,1+12) -- (12,1+12);
\draw [thick] (1,9+12) -- (10.7,9+12);
\draw [thick] (11.3,9+12) -- (12,9+12);
\draw [lightgray!50, fill=lightgray!50] (3,7+12) -- (5,7+12) -- (6,6+12) -- (5,5+12) -- (5,4+12) -- (3,4+12);
\draw [lightgray!50, fill=lightgray!50] (9,7+12) -- (7,7+12) -- (6,6+12) -- (7,5+12) -- (7,4+12) -- (9,4+12) -- (9,5+12) -- (8,5+12) -- (8,6+12) -- (9,6+12);
\draw [lightgray!50, fill=lightgray!50] (10,6+12) -- (10,5+12) -- (9,5+12) -- (9,6+12);
\draw [lightgray!50, fill=lightgray!50] (5,4+12) -- (7,4+12) -- (7,2+12) -- (5,2+12);
\draw [thick] [dotted] (2.5,8+12) -- (2.5,3.5+12) -- (4.5,3.5+12) -- (4.5,2.5+12) --(7.5,2.5+12) -- (7.5,3.5+12) -- (10.5,3.5+12) -- (10.5,8+12) -- (7.5,8+12) -- (7.5,9.5+12) -- (4.5,9.5+12) -- (4.5,8+12) -- (2.5,8+12);
\draw [thick] (5,14) -- (7,14);
\draw [thick] (7,17) -- (7,14);
\draw [thick] (7,17) -- (6.3,17.7);
\draw [thick] (5,14) -- (5,15.7);
\draw [thick] (5,16.3) -- (5,17);
\draw [thick] (7,19) -- (5,17);
\draw [thick] (8,17) -- (10,17);
\draw [thick] (3,16) -- (6.7,16);
\draw [thick] (3,16) -- (3,19);
\draw [thick] (5,19) -- (3,19);
\draw [thick] (5,19) -- (5.7,18.3);
\draw [thick] (7,19) -- (9,19);
\draw [thick] (9,17.3) -- (9,19);
\draw [thick] (9.3,18) -- (10,18);
\draw [thick] (10,17) -- (10,18);
\draw [thick] (8.7,18) -- (8,18);
\draw [thick] (8,17) -- (8,18);
\draw [thick] (9,16.7) -- (9,16);
\draw [thick] (7.3,16) -- (9,16);
\draw [thick] [->] [>=angle 90] (1.5,11.5) -- (0,10.5);
\draw [thick] [->] [>=angle 90] (5.5,5) -- (7.5,5);
\draw [lightgray!50, fill=lightgray!50] (-7+2,0) -- (-7+11,0) -- (-7+11,1) -- (-7+2,1);
\draw [lightgray!50, fill=lightgray!50] (-7+1,1) -- (-7+1,9) -- (-7+2,9) -- (-7+2,1);
\draw [lightgray!50, fill=lightgray!50] (-7+2,9) -- (-7+11,9) -- (-7+11,10) -- (-7+2,10);
\draw [lightgray!50, fill=lightgray!50] (-7+12,9) -- (-7+12,1) -- (-7+11,1) -- (-7+11,9);
\draw [thick] (-7+11,10) -- (-7+11,1.3);
\draw [thick] (-7+11,0) -- (-7+11,0.7);
\draw [thick] (-7+2,10) -- (-7+2,9.3);
\draw [thick] (-7+2,1.3) -- (-7+2,8.7);
\draw [thick] (-7+2,0) -- (-7+2,0.7);
\draw [thick] (-7+1,1) -- (-7+12,1);
\draw [thick] (-7+7,9) -- (-7+10.7,9);
\draw [thick] (-7+1,9) -- (-7+5,9);
\draw [thick] (-7+11.3,9) -- (-7+12,9);
\draw [lightgray!50, fill=lightgray!50] (-4,7) -- (-2,7) -- (-2,6) -- (-1,5) -- (-2,4) -- (-4,4);
\draw [lightgray!50, fill=lightgray!50] (-2,9) -- (-2,7) -- (0,7) -- (0,9);
\draw [lightgray!50, fill=lightgray!50] (-1,5) -- (0,6) -- (0,7) -- (2,7) -- (2,6) -- (1,6) -- (1,5) -- (2,5) -- (2,4) -- (0,4);
\draw [lightgray!50, fill=lightgray!50] (2,5) -- (2,6) -- (3,6) -- (3,5);
\draw [lightgray!50, fill=lightgray!50] (-2,3) -- (0,3) -- (0,2) -- (-2,2);
\draw [thick] [dotted] (-7.5+3,8) -- (-7.5+3,3.5) -- (-7+4.5,3.5) -- (-7+4.5,2.5) -- (-7+7.5,2.5) -- (-7+7.5,3.5) -- (-7+10.5,3.5) -- (-7+10.5,8) -- (-7+7.5,8) -- (-7+7.5,9.5) -- (-7+4.5,9.5) -- (-7+4.5,8) -- (-7.5+3,8);
\draw [thick] (-2,2) -- (0,2);
\draw [thick] (0,3) -- (0,2);
\draw [thick] (0,3) -- (-2,3);
\draw [thick] (-2,2) -- (-2,3);
\draw [thick] (0,6) -- (-2,4);
\draw [thick] (0,6) -- (0,6.7);
\draw [thick] (0,9) -- (0,7.3);
\draw [thick] (-2,9) -- (-2,6);
\draw [thick] (-1.3,5.3) -- (-2,6);
\draw [thick] (-0.7,4.7) -- (0,4);
\draw [thick] (2,4) -- (0,4);
\draw [thick] (2,4.7) -- (2,4);
\draw [thick] (2,5.3) -- (2,7);
\draw [thick] (-1.7,7) -- (2,7);
\draw [thick] (-2.3,7) -- (-4,7);
\draw [thick] (-4,4) -- (-4,7);
\draw [thick] (-4,4) -- (-2,4);
\draw [thick] (1.7,6) -- (1,6) -- (1,5) -- (3,5) -- (3,6) -- (2.3,6);
\draw [lightgray!50, fill=lightgray!50] (14-7+2,0) -- (14-7+11,0) -- (14-7+11,1) -- (14-7+2,1);
\draw [lightgray!50, fill=lightgray!50] (14-7+1,1) -- (14-7+1,9) -- (14-7+2,9) -- (14-7+2,1);
\draw [lightgray!50, fill=lightgray!50] (14-7+2,9) -- (14-7+11,9) -- (14-7+11,10) -- (14-7+2,10);
\draw [lightgray!50, fill=lightgray!50] (14-7+11,9) -- (14-7+11,1) -- (14-7+12,1) -- (14-7+12,9);
\draw [thick] (14-7+11,10) -- (14-7+11,1.3);
\draw [thick] (14-7+11,0) -- (14-7+11,0.7);
\draw [thick] (14-7+2,10) -- (14-7+2,9.3);
\draw [thick] (14-7+2,1.3) -- (14-7+2,8.7);
\draw [thick] (14-7+2,0) -- (14-7+2,0.7);
\draw [thick] (14-7+1,1) -- (14-7+12,1);
\draw [thick] (14-7+7,9) -- (14-7+10.7,9);
\draw [thick] (14-7+1,9) -- (14-7+5,9);
\draw [thick] (14-7+11.3,9) -- (14-7+12,9);
\draw [lightgray!50, fill=lightgray!50] (14-4,7) -- (14-2,7) -- (14-2,6) -- (14-1,5) -- (14-2,4) -- (14-4,4);
\draw [lightgray!50, fill=lightgray!50] (14-2,9) -- (14-2,7) -- (14+0,7) -- (14+0,9);
\draw [lightgray!50, fill=lightgray!50] (14-1,5) -- (14+0,6) -- (14+0,7) -- (14+2,7) -- (14+2,6) -- (14+1,6) -- (14+1,5) -- (14+2,5) -- (14+2,4) -- (14+0,4);
\draw [lightgray!50, fill=lightgray!50] (14+2,5) -- (14+2,6) -- (14+3,6) -- (14+3,5);
\draw [lightgray!50, fill=white] (11.5,6.5) -- (11.5,4.5) -- (10.5,4.5) -- (10.5,6.5);
\draw [thick] (14+0,6) -- (14-2,4);
\draw [thick] (14-2,4) -- (14-4,4);
\draw [thick] (14+0,6) -- (14+0,6.7);
\draw [thick] (14+0,9) -- (14+0,7.3);
\draw [thick] (14-2,9) -- (14-2,6);
\draw [thick] (14-1.3,5.3) -- (14-2,6);
\draw [thick] (14-0.7,4.7) -- (14+0,4);
\draw [thick] (14+2,4) -- (14+0,4);
\draw [thick] (14+2,4.7) -- (14+2,4);
\draw [thick] (14+2,5.3) -- (14+2,7);
\draw [thick] (14-1.7,7) -- (14+2,7);
\draw [thick] (14-2.3,7) -- (14-4,7);
\draw [thick] (14-4,4) -- (14-4,7);
\draw [thick] (14+1.7,6) -- (14+1,6);
\draw [thick] (14+1,6) -- (14+1,5);
\draw [thick] (14+1,5) -- (14+3,5);
\draw [thick] (14+3,5) -- (14+3,6);
\draw [thick] (14+3,6) -- (14+2.3,6);
\draw [thick] (14-3.5,4.5) -- (14-3.5,6.5);
\draw [thick] (14-2.5,6.5) -- (14-3.5,6.5);
\draw [thick] (14-2.5,4.5) -- (14-3.5,4.5);
\draw [thick] (14-2.5,4.5) -- (14-2.5,6.5);
\end{tikzpicture}
\caption{Using an elementary mutation and an ambient isotopy to reduce $\beta_s(D,\sigma)$ while preserving $G^{adj}(D,\sigma)$.}
\label{fignine}
\end{figure}
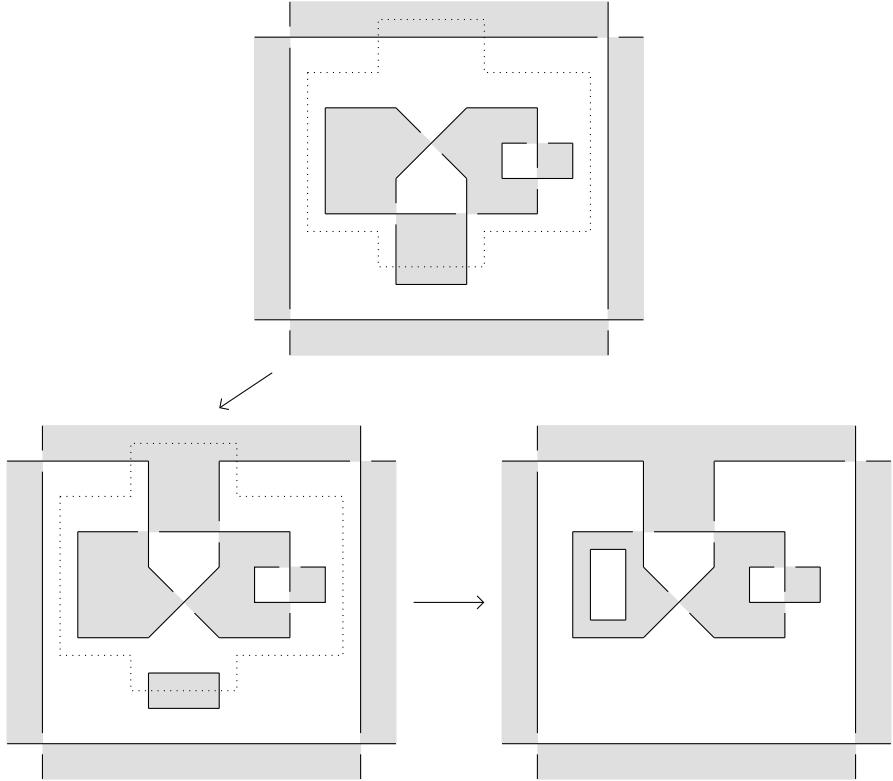

If $D$ is a link diagram that satisfies part 1 of Definition \ref{strongnormaldiagram} but does not satisfy part 2, then it is possible to use ambient isotopies and elementary mutations to produce a link diagram $D'$ with $G^{adj}(D,\sigma)=G^{adj}(D',\sigma')$ and $\beta_s(D',\sigma')=1$. First, find a connected component of $\Gamma_s(D,\sigma)$ that does not include the vertex corresponding to the unbounded complementary region. Then $D$ has a subdiagram $\overline{D}$, completely enclosed in some $R_i$. If $\overline{D}$ has no crossing then it is simply a set of $k \geq 1$ closed curves, some of which may be concentric. We can use an ambient isotopy to reduce $\beta_s(D,\sigma)$ without changing $G^{adj}(D,\sigma)$, by replacing $\overline{D}$ with $k$ small, non-concentric circles, all located in a shaded region of $D$. If $\overline{D}$ has a crossing then using an elementary mutation, we can replace $D$ with a diagram $D'$ in which $\overline{D}$ has been attached to the boundary of $R_i$ with a connected sum, and $R_i$ encloses a small, shaded, crossing-free closed curve instead. (See Figure \ref{fignine} for an example.) As before, we can use an ambient isotopy to move the shaded, crossing-free closed curve into a shaded region, producing an unshaded, crossing-free closed curve. It is easy to see that these manipulations reduce $\beta_s(D,\sigma)$ without affecting $G^{adj}(D,\sigma)$.
\begin{figure} [bht]
\centering
\begin{tikzpicture}
\draw [lightgray!50, fill=lightgray!50] (-1.075736+.424264,.424264+.424264) -- (-1.5-.424264-.424264,.424264+.424264) -- (-1.5,0);
\draw [lightgray!50, fill=lightgray!50] (-1.075736+.424264,-.424264-.424264-1) -- (-1.5-.424264-.424264,-.424264-.424264-1) -- (-1.5-.424264-.424264,-.424264-.424264) -- (-1.5,0) -- (-1.075736+.424264,-.424264-.424264);
\draw [lightgray!50, fill=lightgray!50] (-4.075736+.424264,.424264+.424264) -- (-4.5-.424264-.424264,.424264+.424264) -- (-4.5,0);
\draw [lightgray!50, fill=lightgray!50] (-4.075736+.424264,-.424264-.424264) -- (-4.5-.424264-.424264,-.424264-.424264) -- (-4.5,0);
\draw [lightgray!50, fill=lightgray!50] (-4.075736+.424264,-.424264-.424264) -- (-4.5,-.424264-.424264-.5) -- (-4.5-.424264-.424264,-.424264-.424264);
\draw [lightgray!50, fill=lightgray!50] (-4.075736+.424264,-.424264-.424264-1) -- (-4.5,-.424264-.424264-.5) -- (-4.5-.424264-.424264,-.424264-.424264-1);
\draw [thick] [dotted] [fill=white] (-4.5,0) circle (.6 cm);
\draw [thick] (-1.5-.424264-.424264,-.424264-.424264-1) -- (-1.5-.424264-.424264,-.424264-.424264);
\draw [thick] (-1.075736+.424264,-.424264-.424264-1) -- (-1.075736+.424264,-.424264-.424264);
\draw [thick] (-4.5+.424264+.424264,.424264+.424264) -- (-4.5-.424264-.424264,.424264+.424264);
\draw [thick] (-4.5+.424264,.424264) -- (-4.075736+.424264,.424264+.424264);
\draw [thick] (-1.5+.424264+.424264,.424264+.424264) -- (-1.075736-.424264-.424264-.424264,.424264+.424264);
\draw [thick] (-4.5-.424264,.424264) -- (-4.5-.424264-.424264,.424264+.424264);
\draw [thick] (-4.5-.424264,-.424264) -- (-4.5-.424264-.424264,-.424264-.424264);
\draw [thick] (-4.5+.424264,-.424264) -- (-4.5+.424264+.424264,-.424264-.424264);
\draw [thick] (-4.7,-.35) -- (-4.7,.35);
\draw [thick] (-4.7,-.35) -- (-4.3,-.35);
\draw [thick] [dotted] (-1.5,0) [fill=white] circle (.6 cm);
\draw [thick] (-4.075736+.424264,-.424264-.424264) -- (-4.5-.424264-.424264,-.424264-.424264-1);
\draw [thick] (-4.075736+.424264,-.424264-.424264-1) -- (-4.5-.424264-.424264,-.424264-.424264);
\draw [thick] (-1.5+.424264,.424264) -- (-1.075736+.424264,.424264+.424264);
\draw [thick] (-1.5-.424264,.424264) -- (-1.5-.424264-.424264,.424264+.424264);
\draw [thick] (-1.5-.424264,-.424264) -- (-1.5-.424264-.424264,-.424264-.424264);
\draw [thick] (-1.5+.424264,-.424264) -- (-1.5+.424264+.424264,-.424264-.424264);
\draw [thick] (-1.3,-.35) -- (-1.3,.35);
\draw [thick] (-1.7,-.35) -- (-1.3,-.35);
\end{tikzpicture}
\caption{A kind of ambient isotopy called an $\Omega.4$ move. The direction of rotation is chosen to ``unwind'' the crossing illustrated on the left.}
\label{figten}
\end{figure}
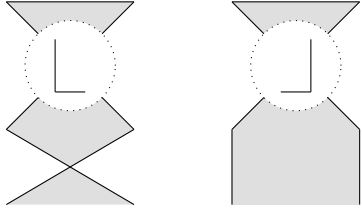

If $C_{ii}$ is nonempty then $D$ has a crossing where an unshaded region $R_i$ is incident twice, as indicated on the left in Figure \ref{figten}. Pictured in the figure is a kind of ambient isotopy called an $\Omega.4$ move. The indicated tangle is not cut out and rotated separately, as it would be in a mutation; instead the entire pictured portion of the link is rotated in 3-space, maintaining the connection with the rest of the link indicated at the bottom of Figure \ref{figten} while ``unwinding'' the pictured crossing. The $\Omega.4$ move produces a diagram $D'$ of the same link type as $D$. It is evident that $G^{adj}(D,\sigma)=G^{adj}(D',\sigma')$, and in $D'$ the crossing from $C_{ii}$ has been eliminated.

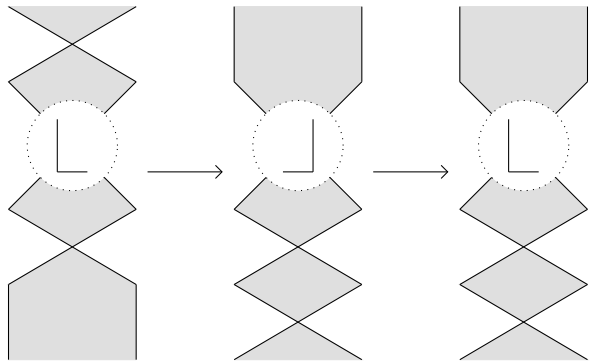
\begin{figure} [bht]
\centering
\begin{tikzpicture}
\draw [lightgray!50, fill=lightgray!50] (-4.5+.424264+.424264,.424264+.424264) -- (-4.5,.424264+.424264+.5) -- (-4.5-.424264-.424264,.424264+.424264) -- (-4.5,0);
\draw [lightgray!50, fill=lightgray!50] (-4.5,.424264+.424264+.5) -- (-4.5-.424264-.424264,1.424264+.424264) -- (-4.5+.424264+.424264,1.424264+.424264);
\draw [lightgray!50, fill=lightgray!50] (-4.075736+.424264,.424264+.424264) -- (-4.5-.424264-.424264,.424264+.424264) -- (-4.5,0);
\draw [lightgray!50, fill=lightgray!50] (-4.075736+.424264,-.424264-.424264-1) -- (-4.5-.424264-.424264,-.424264-.424264-1) -- (-4.075736+.424264,-.424264-.424264) -- (-4.5,0) -- (-4.5-.424264-.424264,-.424264-.424264);
\draw [lightgray!50, fill=lightgray!50] (-4.075736+.424264,-.424264-.424264-1) -- (-4.5-.424264-.424264,-.424264-.424264-1) -- (-4.5-.424264-.424264,-.424264-.424264-2) -- (-4.075736+.424264,-.424264-.424264-2);
\draw [lightgray!50, fill=lightgray!50] (-1.075736+.424264,-.424264-.424264-1) -- (-1.5-.424264-.424264,-.424264-.424264-1) -- (-1.075736+.424264,-.424264-.424264-2) -- (-1.5-.424264-.424264,-.424264-.424264-2);
\draw [lightgray!50, fill=lightgray!50] (-1.075736+.424264,.424264+.424264) -- (-1.5-.424264-.424264,.424264+.424264) -- (-1.5,0);
\draw [lightgray!50, fill=lightgray!50] (-1.075736+.424264,.424264+.424264) -- (-1.5-.424264-.424264,.424264+.424264) -- (-1.5-.424264-.424264,1.424264+.424264) -- (-1.075736+.424264,1.424264+.424264);
\draw [lightgray!50, fill=lightgray!50] (-1.075736+.424264,-.424264-.424264) -- (-1.5-.424264-.424264,-.424264-.424264) -- (-1.5,0);
\draw [lightgray!50, fill=lightgray!50] (-1.075736+.424264,-.424264-.424264) -- (-1.5,-.424264-.424264-.5) -- (-1.5-.424264-.424264,-.424264-.424264);
\draw [lightgray!50, fill=lightgray!50] (-1.075736+.424264,-.424264-.424264-1) -- (-1.5,-.424264-.424264-.5) -- (-1.5-.424264-.424264,-.424264-.424264-1);
\draw [lightgray!50, fill=lightgray!50] (1.5+.424264+.424264,-.424264-.424264-1) -- (1.5-.424264-.424264,-.424264-.424264-1) -- (1.5+.424264+.424264,-.424264-.424264-2) -- (1.5-.424264-.424264,-.424264-.424264-2);
\draw [lightgray!50, fill=lightgray!50] (1.5+.424264+.424264,.424264+.424264) -- (1.5-.424264-.424264,.424264+.424264) -- (1.5,0);
\draw [lightgray!50, fill=lightgray!50] (1.5+.424264+.424264,.424264+.424264) -- (1.5-.424264-.424264,.424264+.424264) -- (1.5-.424264-.424264,1.424264+.424264) -- (1.5+.424264+.424264,1.424264+.424264);
\draw [lightgray!50, fill=lightgray!50] (1.5+.424264+.424264,-.424264-.424264) -- (1.5-.424264-.424264,-.424264-.424264) -- (1.5,0);
\draw [lightgray!50, fill=lightgray!50] (1.5+.424264+.424264,-.424264-.424264) -- (1.5,-.424264-.424264-.5) -- (1.5-.424264-.424264,-.424264-.424264);
\draw [lightgray!50, fill=lightgray!50] (1.5+.424264+.424264,-.424264-.424264-1) -- (1.5,-.424264-.424264-.5) -- (1.5-.424264-.424264,-.424264-.424264-1);
\draw [thick] [dotted] [fill=white] (-4.5,0) circle (.6 cm);
\draw [thick] (-4.5-.424264-.424264,-.424264-.424264-1) -- (-4.5-.424264-.424264,-.424264-.424264-2);
\draw [thick] (-4.075736+.424264,-.424264-.424264-1) -- (-4.075736+.424264,-.424264-.424264-2);
\draw [thick] (-4.075736+.424264,-.424264-.424264-1) -- (-4.5-.424264-.424264,-.424264-.424264);
\draw [thick] (-4.075736+.424264,-.424264-.424264) -- (-4.5-.424264-.424264,-.424264-.424264-1);
\draw [thick] (-4.5+.424264+.424264,.424264+.424264) -- (-4.5-.424264-.424264,1+.424264+.424264);
\draw [thick] (-4.5-.424264-.424264,.424264+.424264) -- (-4.5+.424264+.424264,1+.424264+.424264);
\draw [thick] (-4.5+.424264,.424264) -- (-4.075736+.424264,.424264+.424264);
\draw [thick] (-1.5+.424264+.424264,.424264+.424264) -- (-1.5+.424264+.424264,1.424264+.424264);
\draw [thick] (-1.5-.424264-.424264,1.424264+.424264) -- (-1.075736-.424264-.424264-.424264,.424264+.424264);
\draw [thick] (-4.5-.424264,.424264) -- (-4.5-.424264-.424264,.424264+.424264);
\draw [thick] (-4.5-.424264,-.424264) -- (-4.5-.424264-.424264,-.424264-.424264);
\draw [thick] (-4.5+.424264,-.424264) -- (-4.5+.424264+.424264,-.424264-.424264);
\draw [thick] (-4.7,-.35) -- (-4.7,.35);
\draw [thick] (-4.7,-.35) -- (-4.3,-.35);
\draw [thick] [dotted] (-1.5,0) [fill=white] circle (.6 cm);
\draw [thick] (-1.075736+.424264,-.424264-.424264-2) -- (-1.5-.424264-.424264,-.424264-.424264-1);
\draw [thick] (-1.075736+.424264,-.424264-.424264) -- (-1.5-.424264-.424264,-.424264-.424264-1);
\draw [thick] (-1.5-.424264-.424264,-.424264-.424264) -- (-1.075736+.424264,-.424264-.424264-1);
\draw [thick] (-1.075736+.424264,-.424264-.424264-1) -- (-1.5-.424264-.424264,-.424264-.424264-2);
\draw [thick] (-1.5+.424264,.424264) -- (-1.075736+.424264,.424264+.424264);
\draw [thick] (-1.5-.424264,.424264) -- (-1.5-.424264-.424264,.424264+.424264);
\draw [thick] (-1.5-.424264,-.424264) -- (-1.5-.424264-.424264,-.424264-.424264);
\draw [thick] (-1.5+.424264,-.424264) -- (-1.5+.424264+.424264,-.424264-.424264);
\draw [thick] (-1.3,-.35) -- (-1.3,.35);
\draw [thick] (-1.7,-.35) -- (-1.3,-.35);
\draw [thick] [dotted] (1.5,0) [fill=white] circle (.6 cm);
\draw [thick] (1.5+.424264+.424264,-.424264-.424264-2) -- (1.5-.424264-.424264,-.424264-.424264-1);
\draw [thick] (1.5+.424264+.424264,-.424264-.424264) -- (1.5-.424264-.424264,-.424264-.424264-1);
\draw [thick] (1.5-.424264-.424264,-.424264-.424264) -- (1.5+.424264+.424264,-.424264-.424264-1);
\draw [thick] (1.5+.424264+.424264,-.424264-.424264-1) -- (1.5-.424264-.424264,-.424264-.424264-2);
\draw [thick] (1.5+.424264,.424264) -- (1.5+.424264+.424264,.424264+.424264);
\draw [thick] (1.5+.424264+.424264,1.424264+.424264) -- (1.5+.424264+.424264,.424264+.424264);
\draw [thick] (1.5-.424264,.424264) -- (1.5-.424264-.424264,.424264+.424264);
\draw [thick] (1.5-.424264-.424264,1.424264+.424264) -- (1.5-.424264-.424264,.424264+.424264);
\draw [thick] (1.5-.424264,-.424264) -- (1.5-.424264-.424264,-.424264-.424264);
\draw [thick] (1.5+.424264,-.424264) -- (1.5+.424264+.424264,-.424264-.424264);
\draw [thick] (1.3,-.35) -- (1.3,.35);
\draw [thick] (1.3,-.35) -- (1.7,-.35);
\draw [thick] [->] [>=angle 90] (-0.5,-0.35) -- (0.5,-0.35);
\draw [thick] [->] [>=angle 90] (-3.5,-0.35) -- (-2.5,-0.35);
\end{tikzpicture}
\caption{First, an ambient isotopy called an $\Omega.5$ move transforms $D$ into $D'$. The direction of rotation is chosen to ``unwind'' the upper crossing. Then a mutation of $D'$ yields $D''$, which has the same adjusted Goeritz matrix as $D$.}
\label{figeleven}
\end{figure}

If $i \neq j$ and $C_{ij}$ includes crossings that do not appear on a single two-stranded braid, then there are crossings in $C_{ij}$ separated by a tangle, as illustrated on the left in Figure \ref{figeleven}. A kind of ambient isotopy called an $\Omega.5$ move provides a new diagram $D'$ of the same link type as $D$; $D'$ is pictured in the middle of the figure. As with the $\Omega.4$ move, the portion of the link pictured in Figure \ref{figeleven} is rotated in 3-space without cutting the connections to the rest of the link; in Figure \ref{figeleven} these connections appear at both the top and the bottom. In $D'$ the two crossings from $C_{ij}$ appear consecutively on a two-stranded braid. In general $G^{adj}(D,\sigma) \neq G^{adj}(D',\sigma')$, but an elementary mutation of the second kind produces a diagram $D''$ with $G^{adj}(D'',\sigma'') = G^{adj}(D,\sigma)$. $D''$ is pictured on the right in Figure \ref{figeleven}.

After applying enough of the modifications discussed above, we ultimately obtain a diagram $\widehat{D}$ that is Conway equivalent to $D$, with a shading $\widehat{\sigma}$ which satisfies parts 1, 2, 3 and 5 of Definition \ref{strongnormaldiagram}. If part 4 of Definition \ref{strongnormaldiagram} is not satisfied, then there are consecutive crossings between two unshaded regions, with opposite Goeritz indices. Such a pair of crossings may be removed with an $\Omega.2$ move, as illustrated in Figure \ref{figtwelve}. 

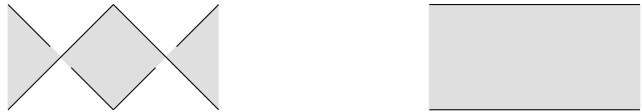
\begin{figure} [bht]
\centering
\begin{tikzpicture} [scale=0.7]
\draw [lightgray!50, fill=lightgray!50] (4,0) -- (4,2) -- (5,1);
\draw [lightgray!50, fill=lightgray!50] (6,0) -- (5,1) -- (6,2) -- (7,1);
\draw [lightgray!50, fill=lightgray!50] (8,0) -- (7,1) -- (8,2);
\draw [lightgray!50, fill=lightgray!50] (12,0) -- (12,2) -- (16,2) -- (16,0);
\draw [thick] (4,0) -- (6,2);
\draw [thick] (6,2) -- (8,0);
\draw [thick] (4,2) -- (4.8,1.2);
\draw [thick] (5.2,0.8) -- (6,0);
\draw [thick] (6,0) -- (6.8,0.8);
\draw [thick] (7.2,1.2) -- (8,2);
\draw [thick] (12,2) -- (16,2);
\draw [thick] (12,0) -- (16,0);
\end{tikzpicture}
\caption{An $\Omega.2$ move.}
\label{figtwelve}
\end{figure}

The result is a diagram $\widehat{D}'$ of the same link type as $\widehat{D}$, which does not include either of the two original crossings. The contributions of the two original crossings to $G(\widehat{D},\widehat{\sigma})$ cancel, so $G(\widehat{D}',\widehat{\sigma}')=G(\widehat{D},\widehat{\sigma})$. Also, it is clear that even though the shaded checkerboard graphs of $(\widehat{D},\widehat{\sigma})$ and $(\widehat{D}',\widehat{\sigma}')$ are not isomorphic, they have the same number of connected components; hence $G^{adj}(\widehat{D},\widehat{\sigma})=G^{adj}(\widehat{D}',\widehat{\sigma}')$.
\end{proof}

Here is a useful property of strongly normal shadings.

\begin{proposition}
\label{simplecon}
If $\sigma$ is a strongly normal shading of a link diagram $D$ then every unshaded complementary region is simply connected.
\end{proposition}

\begin{proof}
Consider an unshaded region $R_i$. As $R_i$ is not the unbounded region, its boundary $\partial R_i$ consists solely of arcs of $D$. Of course each of these arcs separates $R_i$ from a shaded region. If $R_i$ is not simply connected then $\partial R_i$ is not connected, so there are shaded regions $S$ and $S'$ which are incident on different connected components of $\partial R_i$. As $\sigma$ is a strongly normal shading of $D$, $\Gamma_s(D,\sigma)$ is connected; so there is a path connecting the vertices of $\Gamma_s(D,\sigma)$ corresponding to $S$ and $S'$. Such a path provides an arc from a point of $S$ to a point of $S'$, which is entirely contained in the union of the shaded regions and the crossings of $D$. But the existence of such an arc violates the assumption that $S$ and $S'$ are incident on distinct connected components of $\partial R_i$.
\end{proof}

\section{Graph imbeddings and link diagrams}
\label{sec:isotop}

Let $\Gamma$ be a planar graph, i.e., a graph that can be imbedded in the plane, with vertices represented by points and edges represented by pairwise disjoint, piecewise smooth arcs. We say two imbeddings of $\Gamma$ are \emph{locally isotopic} if either imbedding can be obtained from the other through a finite sequence of \emph{local isotopies}, i.e., transformations in which a small portion of the imbedded graph is shifted a small distance, as indicated in Figure~\ref{figthirteen}. (The imbedding is not changed outside the dotted circle.) Notice that the complementary regions of a graph imbedding correspond directly to the complementary regions of a locally isotopic graph imbedding. 

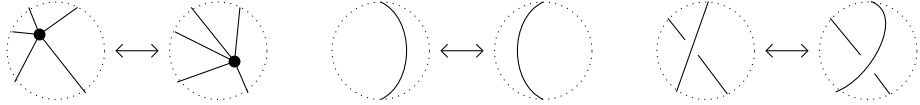
\begin{figure} [bht]
\centering
\begin{tikzpicture} [scale=.72]
\draw [thick] [dotted] (-5,0) circle (.9 cm);
\draw [thick] [dotted] (-2,0) circle (.9 cm);
\draw [thick] [dotted] (1,0) circle (.9 cm);
\draw [thick] [dotted] (4,0) circle (.9 cm);
\draw [thick] [dotted] (7,0) circle (.9 cm);
\draw [thick] [dotted] (10,0) circle (.9 cm);
\draw [thick] [black, fill=black] (-5.3,0.3) circle (.1 cm);
\draw [thick] [black, fill=black] (-1.7,-0.2) circle (.1 cm);
\draw [thick] [<->] [>=angle 90] (-3.9,0) -- (-3.1,0);
\draw [thick] [<->] [>=angle 90] (2.1,0) -- (2.9,0);
\draw [thick] [<->] [>=angle 90] (8.1,0) -- (8.9,0);
\draw [thick] (-5.3,0.3) -- (-5.8,0.35);
\draw [thick] (-5.3,0.3) -- (-5.5,0.8);
\draw [thick] (-5.3,0.3) -- (-4.6,0.8);
\draw [thick] (-5.3,0.3) -- (-4.45,-0.78);
\draw [thick] (-5.3,0.3) -- (-5.76,-0.58);
\draw [thick] (-1.7,-0.2) -- (-2.8,0.35);
\draw [thick] (-1.7,-0.2) -- (-2.5,0.8);
\draw [thick] (-1.7,-0.2) -- (-1.6,0.8);
\draw [thick] (-1.7,-0.2) -- (-1.45,-0.78);
\draw [thick] (-1.7,-0.2) -- (-2.76,-0.58);
\draw [thick] (1,0.9) to [out=335,in=25] (1,-0.9);
\draw [thick] (4,0.9) to [out=205,in=155] (4,-0.9);
\draw [thick] (7.05,0.91) -- (6.46,-0.76);
\draw [thick] (6.86,-0.08) -- (7.4,-0.8);
\draw [thick] (6.3,0.6) -- (6.62,0.2);
\draw [thick] (10.05,0.91) to [out=340,in=20] (9.4,-0.76);
\draw [thick] (10.115,-0.42) -- (10.4,-0.8);
\draw [thick] (9.3,0.6) -- (9.86,-0.08);
\end{tikzpicture}
\caption{Local isotopies of graph imbeddings and link diagrams.}
\label{figthirteen}
\end{figure}

We use the terminology of local isotopy both for graph imbeddings and for link diagrams. (A link diagram is simply a graph imbedding in the plane, with supplementary information that specifies the undercrossing arc at each vertex; a local isotopy preserves the undercrossing information, in addition to making a small change in the diagram.) If two link diagrams are locally isotopic then it is obvious that they represent the same link type. (That is, local isotopy of link diagrams in $\mathbb{R}^2$ implies ambient isotopy of links in $\mathbb{R}^3$.) The reason is that we can build a link in $\mathbb{R}^3$ from a diagram, by attaching small $\cup-$shaped arcs for the underpasses. A local isotopy of the diagram may then be realized by a local perturbation of the associated link in $\mathbb{R}^3$. If a link diagram is equipped with a shading then a locally isotopic diagram inherits the shading directly. 

We use local isotopies to relate link diagrams to graph imbeddings in Theorem~\ref{corr} below. Before stating the theorem, we introduce some notation.

\begin{definition}
A \emph{nonzero integer weighting} of a simple graph $\Gamma$ is a function $w:E(\Gamma) \to \mathbb{Z}$ with $w(e) \neq 0$ $\forall e \in E(\Gamma)$.
\end{definition}

\begin{definition}
Let $\mathcal{IG}$ denote the set of ordered pairs $(w,\lambda)$ where $w$ is a nonzero integer weighting of a planar simple graph $\Gamma$ and $\lambda$ is a local isotopy class of imbeddings of $\Gamma$ in the plane. 
\end{definition}
\begin{definition}
Let $\mathcal{IL}$ denote the set of local isotopy classes of link diagrams with strongly normal shadings. 
\end{definition}
\begin{theorem}
\label{corr}
There is a bijection $\mathcal{D}:\mathcal{IG} \to \mathcal{IL}$ with these properties:
\begin{itemize}
\item If $(w,\lambda) \in \mathcal{IG}$ and $D \in \mathcal{D}(w,\lambda)$ then $D$ is a link diagram with a strongly normal shading $\sigma$ such that $\Gamma_u(D,\sigma)$ is isomorphic to the graph obtained from $\Gamma$ by replacing each edge $e \in E(\Gamma)$ with $|w(e)|$ parallel edges.
\item If $D \in \mathcal{D}(w,\lambda)$ and $e \in E(\Gamma)$ then all of the crossings in $D$ corresponding to $e$ have the same Goeritz index, $\eta=w(e)/|w(e)|$.
\end{itemize}
\end{theorem}
\begin{proof}
Suppose $(w,\lambda) \in \mathcal{IG}$. We build a shaded link diagram from an imbedding included in $\lambda$, as follows. First, replace each vertex with a small, unshaded disk. Then replace each edge with a narrow, unshaded strip. We say ``small'' and ``narrow'' merely to guarantee that there is no intersection between any of the disks and strips, except where a strip representing an edge is attached to a disk representing a vertex incident on that edge. 

A representative of $\mathcal{D}(w,\lambda)$ is constructed from this assembly of disks and strips in two steps. First, erase the two arcs at the ends of each strip, leaving only the two endpoints of each arc. Retain the arcs of the disk boundaries between one strip and the next strip, and the arcs along the sides of the strips. Second, somewhere in each strip, insert a two-stranded braid of $|w(e)|$ crossings across the strip; each of these crossings should have Goeritz index $w(e)/|w(e)|$. See Figure~\ref{figfourteen} for an example. (The disks in the diagram have been drawn as rectangles, and dotted line segments indicate the erased arcs along which strips are attached to disks.)

\begin{figure} [bht]
\centering
\begin{tikzpicture} [scale=.7]
\draw [thick] [black, fill=black] (-3,10) circle (.1 cm);
\draw [thick] [black, fill=black] (-3,6) circle (.1 cm);
\draw [thick] [black, fill=black] (-4,0) circle (.1 cm);
\draw [thick] [black, fill=black] (-2,0) circle (.1 cm);
\draw [thick] [black, fill=black] (-3.5,2.3) circle (.1 cm);
\draw [thick] [black, fill=black] (-2.5,2.3) circle (.1 cm);
\draw [thick] [black, fill=black] (-3,4.5) circle (.1 cm);
\draw [thick] (-3,10) -- (-3,6);
\draw [thick] (-4,0) -- (-2,0);
\draw [thick] (-3.5,2.3) -- (-2.5,2.3);
\draw [thick] (-3,6) to [out=230,in=90] (-4,0);
\draw [thick] (-3,6) to [out=310,in=90] (-2,0);
\node at (-3.5,8) {-3};
\node at (-3,2.8) {2};
\node at (-3,0.4) {2};
\node at (-4.1,5.1) {-1};
\node at (-1.9,5.1) {-2};
\node at (-3.5,10) {$a$};
\node at (-3.5,6.2) {$b$};
\node at (-3.02,4.1) {$c$};
\node at (-3.5,1.8) {$d$};
\node at (-2.5,1.75) {$e$};
\node at (-4,-0.5) {$f$};
\node at (-2,-0.5) {$g$};
\node at (5.98,9.77) {$a$};
\node at (5.98,7.4) {$b$};
\node at (4.8,4) {$d$};
\node at (7.2,3.92) {$e$};
\node at (2,1) {$f$};
\node at (10.5,1) {$g$};
\draw [lightgray!50, fill=lightgray!50] (12,5) -- (12.5,5) -- (12.5,5.2) -- (7,8.5) -- (7,8);
\draw [lightgray!50, fill=lightgray!50] (16/3,9) -- (5,28/3) -- (5,26/3);
\draw [lightgray!50, fill=lightgray!50] (18/3,9) -- (17/3,28/3) -- (16/3,9) -- (17/3,26/3);
\draw [lightgray!50, fill=lightgray!50] (20/3,9) -- (19/3,28/3) -- (18/3,9) -- (19/3,26/3);
\draw [lightgray!50, fill=lightgray!50] (22/3,9) -- (21/3,28/3) -- (20/3,9) -- (21/3,26/3);
\draw [lightgray!50, fill=lightgray!50] (7.5,10) -- (4.5,10) -- (4.5,10.5) -- (7.5,10.5);
\draw [lightgray!50, fill=lightgray!50] (7,10.5) -- (7,8) -- (7.5,7.7) -- (7.5,10.5);
\draw [lightgray!50, fill=lightgray!50] (5,10.5) -- (4.5,10.5) -- (4.5,8) -- (5,8);
\draw [thick] (7,10) -- (5,10);
\draw [thick] [dotted] (7,9.55) -- (5,9.55);
\draw [thick] [dotted] (7,8) -- (5,8);
\draw [thick] [dotted] (7,8) -- (7,6.8);
\draw [thick] [dotted] (5,8) -- (5,6.8);
\draw [thick] [dotted] (1,2) -- (2,2);
\draw [thick] [dotted] (3,0) -- (3,2);
\draw [thick] [dotted] (9,0) -- (9,2);
\draw [thick] [dotted] (12,2) -- (10,2);
\draw [thick] [dotted] (5.1,3) -- (5.1,5);
\draw [thick] [dotted] (6.9,3) -- (6.9,5);
\draw [thick] (5,28/3) -- (5,10);
\draw [thick] (7,28/3) -- (7,10);
\draw [thick] (5,26/3) -- (5,8);
\draw [thick] (7,26/3) -- (7,8);
\draw [thick] (5,28/3) -- (21/4-1/20,109/12+1/20);
\draw [thick] (17/3,26/3) -- (1/20+65/12,-1/20+107/12);
\draw [thick] (5,26/3) -- (17/3,28/3);
\draw [thick] (5+2/3,28/3) -- (-1/20+21/4+2/3,1/20+109/12);
\draw [thick] (17/3+2/3,26/3) -- (1/20+65/12+2/3,-1/20+107/12);
\draw [thick] (5+2/3,26/3) -- (17/3+2/3,28/3);
\draw [thick] (5+4/3,28/3) -- (-1/20+21/4+4/3,1/20+109/12);
\draw [thick] (17/3+4/3,26/3) -- (1/20+65/12+4/3,-1/20+107/12);
\draw [thick] (5+4/3,26/3) -- (17/3+4/3,28/3);
\draw [lightgray!50, fill=lightgray!50] (12,0) -- (12,5) -- (12.5,5) -- (12.5,0);
\draw [lightgray!50, fill=lightgray!50] (0.5,0) -- (0.5,-0.5) -- (12.5,-0.5) -- (12.5,0);
\draw [lightgray!50, fill=lightgray!50] (1,-0.5) -- (0.5,-0.5) -- (0.5,4) -- (1,4);
\draw [lightgray!50, fill=lightgray!50] (6,0.5) -- (7,-0.5) -- (5,-0.5);
\draw [lightgray!50, fill=lightgray!50] (6,0.5) -- (6.5,1) -- (6,1.5) -- (5.5,1);
\draw [lightgray!50, fill=lightgray!50] (5,2.5) -- (7,2.5) -- (6,1.5);
\draw [lightgray!50, fill=lightgray!50] (2,3) -- (10,3) -- (10,2) -- (2,2);
\draw [lightgray!50, fill=lightgray!50] (2,2) -- (2,4) -- (60/19,42/19+10/3) -- (5,6.8) -- (4.5,5) -- (4.5,2);
\draw [thick] (12,0) -- (12,3);
\draw [thick] (12,0) -- (6.5,0);
\draw [thick] (1,0) -- (5.5,0);
\draw [thick] (1,2) -- (1,0);
\draw [thick] (5.5,0) -- (6.5,1);
\draw [thick] (6.5,0) -- (6.2,0.3);
\draw [thick] (5.5,1) -- (5.8,0.7);
\draw [thick] (6.5,1) -- (6.2,1.3);
\draw [thick] (5.5,2) -- (5.8,1.7);
\draw [thick] (5.5,1) -- (6.5,2);
\draw [thick] (10,2) -- (6.5,2);
\draw  [thick](5.5,2) -- (2,2);
\draw [lightgray!50, fill=lightgray!50] (6,3.5) -- (7,2.5) -- (5,2.5);
\draw [lightgray!50, fill=lightgray!50] (6,3.5) -- (6.5,4) -- (6,4.5) -- (5.5,4);
\draw [lightgray!50, fill=lightgray!50] (5,5.5) -- (7,5.5) -- (6,4.5);
\draw [lightgray!50, fill=lightgray!50] (7.5,5) -- (10,5) -- (10,2.5) -- (7.5,2.5);
\draw [lightgray!50, fill=lightgray!50] (10,5) -- (7.5,6.5) -- (7.5,2.5);
\draw [lightgray!50, fill=lightgray!50] (5,6.8) -- (7,6.8) -- (10,5) -- (2.8,5);
\draw [lightgray!50, fill=lightgray!50] (4.5,8) -- (5,8) -- (60/19,42/19+10/3) -- (5,6.8) -- (1,4) -- (0.5,4);
\draw [thick] [fill=white] (5.7,5.55) -- (5.7,6.15) -- (6.3,6.15) -- (6.3,5.55) -- (5.7,5.55);
\node at (5.98,5.85) {$c$};
\draw [thick] (5.5,3) -- (6.5,4);
\draw [thick] (6.5,3) -- (6.2,3.3);
\draw [thick] (5.5,4) -- (5.8,3.7);
\draw [thick] (6.5,4) -- (6.2,4.3);
\draw [thick] (5.5,5) -- (5.8,4.7);
\draw [thick] (5.5,4) -- (6.5,5);
\draw [thick] (7.5,5) -- (6.5,5);
\draw [thick] (7.5,5) -- (7.5,3);
\draw [thick] (6.5,3) -- (7.5,3);
\draw [thick] (5.5,3) -- (4.5,3);
\draw [thick] (4.5,5) -- (4.5,3);
\draw [thick] (4.5,5) -- (5.5,5);
\draw [lightgray!50, fill=lightgray!50] (10,3) -- (10.5,3.5) -- (10,4);
\draw [lightgray!50, fill=lightgray!50] (11,3) -- (10.5,3.5) -- (11,4) -- (11.5,3.5);
\draw [lightgray!50, fill=lightgray!50] (12.5,4.5) -- (12.5,2.5) -- (11.5,3.5);
\draw [thick] (7,8) -- (12,5);
\draw [thick] (7,6.8) -- (10,5);
\draw [thick] (7,6.8) -- (5,6.8);
\draw [thick] (10,3) -- (10,2);
\draw [thick] (10,3) -- (11,4);
\draw [thick] (10,4) -- (10.3,3.7);
\draw [thick] (11,3) -- (10.7,3.3);
\draw [thick] (11,3) -- (12,4);
\draw [thick] (11,4) -- (11.3,3.7);
\draw [thick] (12,3) -- (11.7,3.3);
\draw [thick] (12,4) -- (12,5);
\draw [thick] (10,4) -- (10,5);
\draw [thick] (5,6.8) -- (1,4);
\draw [thick] (1,2) -- (1,4);
\draw [thick] (2.75,5) -- (2,4);
\draw [thick] (3.5,6) -- (5,8);
\draw [thick] (2,2) -- (2,4);
\end{tikzpicture}
\caption{On the left is a plane imbedding of a simple graph $\Gamma$. The isotopy class is $\lambda$, and a nonzero integer weighting $w$ is indicated. On the right is a diagram from $\mathcal{D}(w,\lambda)$. Dotted segments indicate arcs where strips (which correspond to edges) attach to disks (which correspond to vertices).}
\label{figfourteen}
\end{figure}
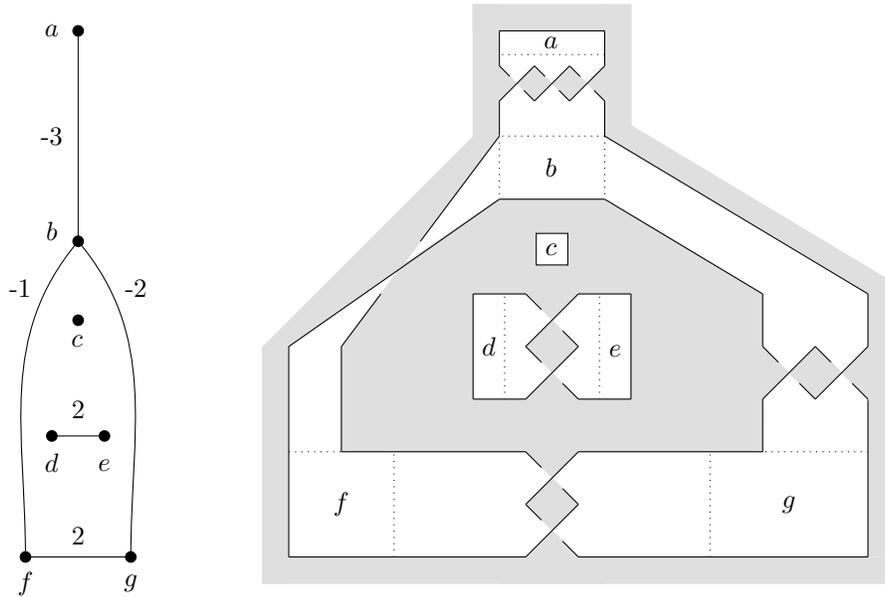

It is clear that if a local isotopy is applied to the imbedding used in the construction, the resulting link diagram will also have a local isotopy applied to it. That is, we have a well-defined function $\mathcal{D}:\mathcal{IG} \to \mathcal{IL}$.

It is not difficult to construct the inverse bijection $\mathcal{D}^{-1}:\mathcal{IL} \to \mathcal{IG}$. Suppose $D$ is a link diagram with a strongly normal shading $\sigma$. Let $\Gamma$ be the simple graph obtained from $\Gamma_u(D,\sigma)$ by replacing each nonempty set of parallel edges with a single edge. Then there are a local isotopy class $\lambda$ of imbeddings of $\Gamma$ in the plane and a nonzero integer weighting $w$ of $\Gamma$, defined as follows. Place a vertex in each unshaded region $R_i$. When $R_i$ and $R_j$ are unshaded regions that share one or more crossings, draw one edge $e_{ij}$ from the $R_i$ vertex to the $R_j$ vertex, and define $w(e_{ij})$ to be the total of the Goeritz indices of the crossings shared by $R_i$ and $R_j$. This edge $e_{ij}$ should be completely contained in $R_i \cup R_j$, except for one point which is a shared crossing. The fact that the crossings between $R_i$ and $R_j$ occur on a single two-stranded braid guarantees that up to local isotopy, it does not matter which crossing appears on $e_{ij}$. If $j \neq k$ and $e_{ij},e_{ik}$ are both edges, then $e_{ij} \cap e_{ik}$ should include only the $R_i$ vertex. Proposition~\ref{simplecon} implies that there is only one local isotopy class of such imbeddings of $\Gamma$, so we have constructed a well-defined function $\mathcal{IL} \to \mathcal{IG}$. It is clear that both compositions of this function with $\mathcal{D}$ are identity maps.
\end{proof}

\section{The imbeddings of a planar graph}
\label{sec:imbed}

In this section we briefly discuss the connections among the various imbeddings of a given planar graph. The original theory is due to Whitney~\cite{W1,W}. Ore gave a very thorough account in his book on what was then the four-color conjecture~\cite{O}; see also Greene~\cite{Gre}, Lipson~\cite{Lip}, and Mohar and Thomassen \cite{MT}. 

There are four kinds of operations that connect all the imbeddings of a planar graph $\Gamma$ to each other. We have already encountered the simplest of the four, local isotopy. The second kind of operation is also very simple: if $\Gamma$ is a disconnected graph then we may alter an imbedding of $\Gamma$ in the plane by moving the imbedded image of a connected component $C$ of $\Gamma$ from one complementary region to another. If necessary, we can use local isotopies to reduce the size of the imbedded image of $C$ before moving it into a different region.

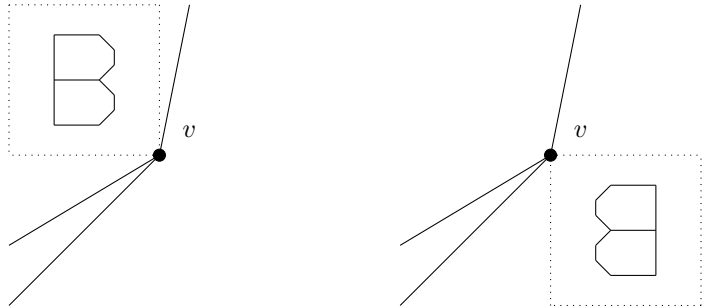
\begin{figure} [bht]
\centering
\begin{tikzpicture} [scale=.4]
\node at (6,0.8) {$v$};
\draw [thick] [black, fill=black] (5,0) circle (.2 cm);
\draw [thick] [dotted] (5,0) -- (0,0);
\draw [thick] [dotted] (0,5) -- (0,0);
\draw [thick] [dotted] (0,5) -- (5,5);
\draw [thick] [dotted] (5,5) -- (5,0);
\draw [thick] (5,0) -- (6,5); 
\draw [thick] (5,0) -- (0,-3);
\draw [thick] (5,0) -- (0,-5);
\draw [thick] (1.5,1) -- (1.5,4);
\draw [thick] (3,4) -- (1.5,4);
\draw [thick] (3,4) -- (3.5,3.5);
\draw [thick] (3.5,3) -- (3.5,3.5);
\draw [thick] (3.5,3) -- (3,2.5);
\draw [thick] (1.5,2.5) -- (3,2.5);
\draw [thick] (3.5,2) -- (3,2.5);
\draw [thick] (3.5,2) -- (3.5,1.5);
\draw [thick] (3,1) -- (3.5,1.5);
\draw [thick] (3,1) -- (1.5,1);
\node at (19,0.8) {$v$};
\draw [thick] [black, fill=black] (18,0) circle (.2 cm);
\draw [thick] [dotted] (18,0) -- (23,0);
\draw [thick] [dotted] (23,-5) -- (23,0);
\draw [thick] [dotted] (23,-5) -- (18,-5);
\draw [thick] [dotted] (18,-5) -- (18,0);
\draw [thick] (18,0) -- (19,5); 
\draw [thick] (18,0) -- (13,-3);
\draw [thick] (18,0) -- (13,-5);
\draw [thick] (21.5,-1) -- (21.5,-4);
\draw [thick] (21.5,-4) -- (20,-4);
\draw [thick] (19.5,-3.5) -- (20,-4);
\draw [thick] (19.5,-3.5) -- (19.5,-3);
\draw [thick] (20,-2.5) -- (19.5,-3);
\draw [thick] (20,-2.5) -- (21.5,-2.5);
\draw [thick] (20,-2.5) -- (19.5,-2);
\draw [thick] (19.5,-1.5) -- (19.5,-2);
\draw [thick] (19.5,-1.5) -- (20,-1);
\draw [thick] (21.5,-1) -- (20,-1);
\end{tikzpicture}
\caption{Swinging around a vertex.}
\label{figfifteen}
\end{figure}

The third kind of operation might be called ``swinging.'' Suppose $v \in V(\Gamma)$ and there is a disk in the plane whose boundary intersects an imbedding of $\Gamma$ only at $v$. Then we may swing the portion of the imbedding contained in the disk around $v$ from one complementary region to another, as indicated in Figure~\ref{figfifteen}. Once again, if necessary we can use local isotopies to change the shape or size of the disk before swinging it into a different region.

The fourth kind of imbedding operation is a flip. If $v,w$ are two points and there is a disk whose boundary intersects the image of the imbedding in a subset of $\{v,w\}$ then the portion of the imbedding inside the disk may be rotated through an angle of $\pi$, with the line through $v$ and $w$ as axis. (It may be necessary to use local isotopies first, to make sure the disk is symmetric with respect to the line through $v$ and $w$.) See Figure~\ref{figsixteen}, where it is assumed that $v,w \in V(\Gamma)$. This assumption is not necessary. For instance we may flip a disk around $\{v,w\}$ if $v$ is not included in the imbedded image of $\Gamma$, so long as the intersection of the disk boundary with the image of $\Gamma$ is $\emptyset$ or $\{w\}$. We may also flip around $\{v,w\}$ if $v$ or $w$ is a point on an edge of $\Gamma$.

\begin{figure} [bht]
\centering
\begin{tikzpicture} [scale=.4]
\draw [thick] [dotted] (-2.5,2.475) circle (2.5 cm);
\node at (-2.5,5.75) {$v$};
\node at (-2.5,-0.8) {$w$};
\draw [thick] [black, fill=black] (-2.5,0) circle (.2 cm);
\draw [thick] [black, fill=black] (-2.5,5) circle (.2 cm);
\draw [thick] (-2.5,5) -- (-1,6); 
\draw [thick] (-2.5,5) -- (0,6); 
\draw [thick] (-2.5,5) -- (-4,6);
\draw [thick] (-2.5,0) -- (-5,-1);
\draw [thick] (-2.5,0) -- (-4,-1);
\draw [thick] (-3.5,1) -- (-3.5,4);
\draw [thick] (-2,4) -- (-3.5,4);
\draw [thick] (-2,4) -- (-1.5,3.5);
\draw [thick] (-1.5,3) -- (-1.5,3.5);
\draw [thick] (-1.5,3) -- (-2,2.5);
\draw [thick] (-3.5,2.5) -- (-2,2.5);
\draw [thick] (-1.5,2) -- (-2,2.5);
\draw [thick] (-1.5,2) -- (-1.5,1.5);
\draw [thick] (-2,1) -- (-1.5,1.5);
\draw [thick] (-2,1) -- (-3.5,1);
\node at (12.5,5.75) {$v$};
\node at (12.5,-0.8) {$w$};
\draw [thick] [black, fill=black] (12.5,0) circle (.2 cm);
\draw [thick] [black, fill=black] (12.5,5) circle (.2 cm);
\draw [thick] [dotted] (12.5,2.475) circle (2.5 cm);
\draw [thick] (12.5,5) -- (14,6); 
\draw [thick] (12.5,5) -- (11,6);
\draw [thick] (12.5,5) -- (15,6);
\draw [thick] (12.5,0) -- (10,-1);
\draw [thick] (12.5,0) -- (11,-1);
\draw [thick] (13.5,1) -- (13.5,4);
\draw [thick] (13.5,1) -- (12,1);
\draw [thick] (12,1) -- (11.5,1.5);
\draw [thick] (11.5,1.5) -- (11.5,2);
\draw [thick] (11.5,2) -- (12,2.5);
\draw [thick] (12,2.5) -- (13.5,2.5);
\draw [thick] (12,2.5) -- (11.5,3);
\draw [thick] (11.5,3) -- (11.5,3.5);
\draw [thick] (11.5,3.5) -- (12,4);
\draw [thick] (12,4) -- (13.5,4);
\end{tikzpicture}
\caption{Flipping with respect to two vertices.}
\label{figsixteen}
\end{figure}
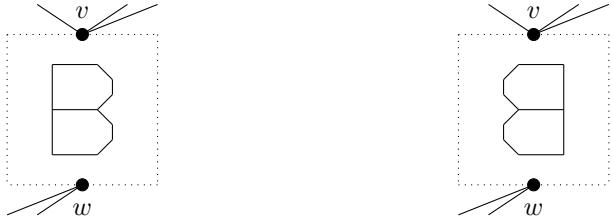

In general, there is not a one-to-one correspondence between the complementary regions of the imbeddings involved in a flip. The flip's effect is to perform a Whitney twist on the geometric dual graph, building two complementary regions of the new imbedding from pieces of two complementary regions of the original imbedding. However, if the portion of $\Gamma$ outside the disk is a single edge $e$, then there is a direct correspondence between the complementary regions of the two imbeddings. If we compose the flip with a reflection (which is also a special kind of flip), and think of the imbeddings on the 2-sphere rather than the plane, then the effect of the two flips is the same as the effect of a non-planar local isotopy, which moves $e$ across the point at infinity. See Figure \ref{figsixteena} for an example.

\begin{figure} [bht]
\centering
\begin{tikzpicture} [scale=.4]
\draw [thick] [dotted] (-2.5,2.475) circle (2.5 cm);
\draw [thick] [black, fill=black] (-2.5,0) circle (.2 cm);
\draw [thick] [black, fill=black] (-4,3.5) circle (.2 cm);
\draw [thick] [black, fill=black] (-4,1.5) circle (.2 cm);
\draw [thick] [black, fill=black] (-2.5,5) circle (.2 cm);
\draw [thick] [black, fill=black] (-2.5,2.5) circle (.2 cm);
\draw [thick] (-2.5,0) -- (-4,3.5); 
\draw [thick] (-2.5,5) -- (-4,3.5); 
\draw [thick] (-4,3.5) -- (-4,1.5); 
\draw [thick] (-2.5,0) -- (-4,1.5);
\draw [thick] (-2.5,0) -- (-2.5,5);
\draw [thick] (-2.5,0) to [out=180, in=270] (-6,2.5);
\draw [thick] (-2.5,5) to [out=180, in=90] (-6,2.5);
\draw [thick] (5.5,0) to [out=180, in=270] (2,2.5);
\draw [thick] (5.5,5) to [out=180, in=90] (2,2.5);
\draw [thick] [black, fill=black] (5.5,0) circle (.2 cm);
\draw [thick] [black, fill=black] (5.5,5) circle (.2 cm);
\draw [thick] [black, fill=black] (7,3.5) circle (.2 cm);
\draw [thick] [black, fill=black] (7,1.5) circle (.2 cm);
\draw [thick] [black, fill=black] (5.5,2.5) circle (.2 cm);
\draw [thick] (5.5,0) -- (7,3.5); 
\draw [thick] (5.5,5) -- (7,3.5); 
\draw [thick] (7,3.5) -- (7,1.5); 
\draw [thick] (5.5,0) -- (7,1.5);
\draw [thick] (5.5,0) -- (5.5,5);
\draw [thick] (12.5,0) to [out=180, in=270] (9,2.5);
\draw [thick] (12.5,5) to [out=180, in=90] (9,2.5);
\draw [thick] [black, fill=black] (12.5,0) circle (.2 cm);
\draw [thick] [black, fill=black] (12.5,5) circle (.2 cm);
\draw [thick] [dotted] (12.5,2.475) circle (4 cm);
\draw [thick] [black, fill=black] (14,3.5) circle (.2 cm);
\draw [thick] [black, fill=black] (14,1.5) circle (.2 cm);
\draw [thick] [black, fill=black] (12.5,2.5) circle (.2 cm);
\draw [thick] (12.5,5) -- (14,3.5); 
\draw [thick] (12.5,0) -- (14,3.5);
\draw [thick] (14,3.5) -- (14,1.5); 
\draw [thick] (12.5,0) -- (14,1.5);
\draw [thick] (12.5,0) -- (12.5,5);
\draw [thick] [black, fill=black] (20,0) circle (.2 cm);
\draw [thick] [black, fill=black] (18.5,3.5) circle (.2 cm);
\draw [thick] [black, fill=black] (18.5,1.5) circle (.2 cm);
\draw [thick] [black, fill=black] (20,5) circle (.2 cm);
\draw [thick] [black, fill=black] (20,2.5) circle (.2 cm);
\draw [thick] (20,5) -- (18.5,3.5); 
\draw [thick] (20,0) -- (18.5,3.5);
\draw [thick] (18.5,3.5) -- (18.5,1.5); 
\draw [thick] (20,0) -- (18.5,1.5);
\draw [thick] (20,0) -- (20,5);
\draw [thick] (20,0) to [out=0, in=270] (23.5,2.5);
\draw [thick] (20,5) to [out=0, in=90] (23.5,2.5);

\end{tikzpicture}
\caption{Two flips combine to move an edge across the point at infinity.}
\label{figsixteena}
\end{figure}

\section{Goeritz equivalence $\implies$ Conway equivalence}
\label{sec:gc}

To complete the proof of Theorem \ref{main}, we need to verify that if $D$ and $D^{\prime }$ are link diagrams with the same adjusted Goeritz matrix, then $D$ can be transformed into $D^{\prime }$ using Reidemeister moves and elementary mutations. 

According to Proposition~\ref{goeritzprop}, we may presume that $G^{adj}(D,\sigma)=G^{adj}(D',\sigma')$ and $\sigma,\sigma^{\prime }$ are both strongly normal shadings. The matrix equality implies that $\Gamma_u(D,\sigma) \cong \Gamma_u(D',\sigma')$, and the isomorphism preserves the Goeritz indices of the crossings. Let $\Gamma$ be an abstract simple graph obtained from either $\Gamma_u(D,\sigma)$ or $\Gamma_u(D',\sigma')$ by replacing each nonempty family of parallel edges with a single edge, and let $w$ be the nonzero integer weighting of $\Gamma$ such that for each edge $e$, $w(e)$ is the total of the Goeritz indices of the crossings corresponding to $e$. 

According to Theorem~\ref{corr}, there are isotopy classes $\lambda,\lambda'$ of imbeddings of $\Gamma$ in the plane such that $\mathcal{D}(w,\lambda)$ and $\mathcal{D}(w,\lambda')$ include $D$ and $D'$ (respectively). It follows that $\lambda$ and $\lambda'$ are related through a finite sequence of the operations of the preceding section. The first type of operation, a local isotopy, has no effect on the link type of the associated diagram. The second type of operation, moving the imbedded image of a connected component from one complementary region to another, also has no effect on link type. 

The third type of graph operation involves swinging a disk around a vertex $v$, where $v$ provides the only connection between the portion of $\Gamma$ inside the disk and the portion of $\Gamma$ outside the disk. The corresponding operation on link diagrams involves changing the location on the boundary of an unshaded region where a connected sum is performed. It is easy to make such a change using an elementary mutation of the third kind. (If the two arcs where the connected summand is attached are parts of the same link component, the swinging can be performed with an ambient isotopy.) See Figure~\ref{figseventeen} for an example.

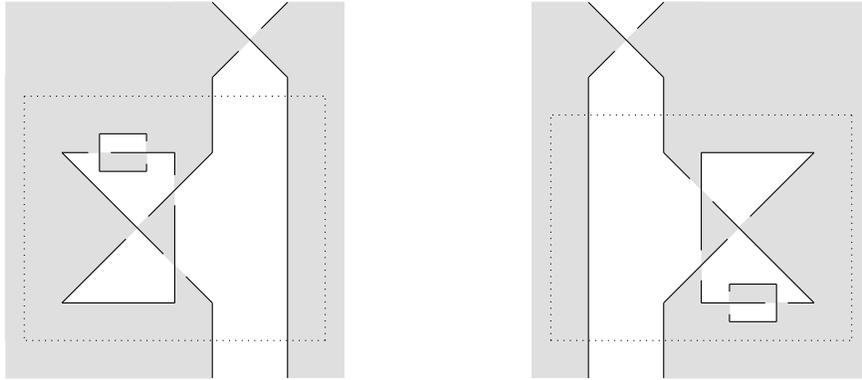
\begin{figure} [bht]
\centering
\begin{tikzpicture} [scale=.5]
\draw [lightgray!50, fill=lightgray!50] (-7,5) -- (-8,4) -- (-7,3);
\draw [lightgray!50, fill=lightgray!50] (-7,5) -- (-7,10) -- (-6,10) -- (-6,6);
\draw [lightgray!50, fill=lightgray!50] (-11.5,10) -- (-6,10) -- (-6,6) -- (-7.75,6) -- (-7.75,6.5) -- (-9,6.5) -- (-9,6) -- (-11.5,6);
\draw [lightgray!50, fill=lightgray!50] (-7.75,6) -- (-7.75,5.5) -- (-9,5.5) -- (-9,6);
\draw [lightgray!50, fill=lightgray!50] (-6,10) -- (-6,8) -- (-5,9);
\draw [lightgray!50, fill=lightgray!50] (-2.5,0) -- (-2.5,10) -- (-4,10) -- (-5,9) -- (-4,8) -- (-4,0);
\draw [lightgray!50, fill=lightgray!50] (-11.5,0) -- (-6,0) -- (-6,2) -- (-7,3) -- (-7,2) -- (-11.5,2);
\draw [lightgray!50, fill=lightgray!50] (-11.5,1) -- (-11.5,8) -- (-11,7) -- (-8,4) -- (-11,1);
\draw [thick] [dotted] (-3,1) -- (-11,1);
\draw [thick] [dotted] (-3,1) -- (-3,7.5);
\draw [thick] [dotted] (-11,7.5) -- (-3,7.5);
\draw [thick] [dotted] (-11,7.5) -- (-11,1);
\draw [thick] (-9,6.5) -- (-7.75,6.5);
\draw [thick] (-9,6.5) -- (-9,5.5);
\draw [thick] (-9,5.5) -- (-7.75,5.5);
\draw [thick] (-4,0) -- (-4,8); 
\draw [thick] (-6,10) -- (-4,8); 
\draw [thick] (-6,8) -- (-5.3,8.7); 
\draw [thick] (-4,10) -- (-4.7,9.3); 
\draw [thick] (-6,8) -- (-6,6); 
\draw [thick] (-6,6) -- (-7.7,4.3); 
\draw [thick] (-10,2) -- (-8.3,3.7); 
\draw [thick] (-10,2) -- (-7,2); 
\draw [thick] (-7,4.6) -- (-7,2); 
\draw [thick] (-7,6) -- (-7,5.4); 
\draw [thick] (-10,6) -- (-9.3,6); 
\draw [thick] (-7,6) -- (-8.7,6); 
\draw [thick] (-7.75,6.3) -- (-7.75,6.5); 
\draw [thick] (-7.75,5.7) -- (-7.75,5.5); 
\draw [thick] (-10,6) -- (-7.3,3.3); 
\draw [thick] (-6,2) -- (-6.7,2.7); 
\draw [thick] (-6,2) -- (-6,0); 
\draw [lightgray!50, fill=lightgray!50] (7,5) -- (8,4) -- (7,3);
\draw [lightgray!50, fill=lightgray!50] (7,5) -- (7,10) -- (6,10) -- (6,6);
\draw [lightgray!50, fill=lightgray!50] (11.5,10) -- (6,10) -- (6,6) -- (11.5,6);
\draw [lightgray!50, fill=lightgray!50] (6,10) -- (6,8) -- (5,9);
\draw [lightgray!50, fill=lightgray!50] (2.5,0) -- (2.5,10) -- (4,10) -- (5,9) -- (4,8) -- (4,0);
\draw [lightgray!50, fill=lightgray!50] (11.5,0) -- (6,0) -- (6,2) -- (7,3) -- (7,2) -- (7.75,2) -- (7.75,1.5) -- (9,1.5) -- (9,2) -- (11.5,2);
\draw [lightgray!50, fill=lightgray!50] (7.75,2) -- (7.75,2.5) -- (9,2.5) -- (9,2);
\draw [lightgray!50, fill=lightgray!50] (11.5,1) -- (11.5,8) -- (11,7) -- (8,4) -- (11,1);
\draw [thick] [dotted] (3,1) -- (11,1);
\draw [thick] [dotted] (3,1) -- (3,7);
\draw [thick] [dotted] (11,7) -- (3,7);
\draw [thick] [dotted] (11,7) -- (11,1);
\draw [thick] (7.75,2.3) -- (7.75,2.5); 
\draw [thick] (7.75,1.7) -- (7.75,1.5); 
\draw [thick] (9,2.5) -- (7.75,2.5);
\draw [thick] (9,2.5) -- (9,1.5);
\draw [thick] (9,1.5) -- (7.75,1.5); 
\draw [thick] (7,6) -- (10,6); 
\draw [thick] (4,0) -- (4,8); 
\draw [thick] (6,10) -- (5.3,9.3); 
\draw [thick] (4,8) -- (4.7,8.7); 
\draw [thick] (4,10) -- (6,8); 
\draw [thick] (6,8) -- (6,6); 
\draw [thick] (6,6) -- (6.7,5.3); 
\draw [thick] (10,2) -- (7.3,4.7); 
\draw [thick] (10,2) -- (9.3,2); 
\draw [thick] (7,2) -- (8.7,2); 
\draw [thick] (7,2.6) -- (7,2); 
\draw [thick] (7,3.4) -- (7,6); 
\draw [thick] (10,6) -- (8.3,4.3); 
\draw [thick] (6,2) -- (7.7,3.7); 
\draw [thick] (6,2) -- (6,0); 
\end{tikzpicture}
\caption{An elementary mutation of the third kind swings a subgraph of $\Gamma_u$.}
\label{figseventeen}
\end{figure}

Similarly, the fourth kind of graph imbedding operation, a flip with respect to $v$ and $w$, corresponds to an elementary mutation of the second kind. (Again, an ambient isotopy may suffice in a simple case, e.g., if $v$ lies in a shaded region.)

\section{Oriented links}
\label{sec:or}

In this section we adapt the discussion of Sections 2 -- 7 for oriented link diagrams. The Goeritz index is replaced with the~\emph{checkerboard writhe}, denoted $\eta_{or}$. The checkerboard writhe determines both the Goeritz index and the writhe, and also has the property that an underpass-overpass reversal multiplies $\eta_{or}$ by $-1$. See Figures~\ref{figeighteen} and~\ref{fignineteen}.

\begin{figure} [bht]
\centering
\begin{tikzpicture} [>=angle 90]
\draw [thick] [->] (-2.5,1) -- (-1/2,-1);
\draw [thick] [<-] (-2.5,-1) -- (-1.6,-.1);
\draw [thick] (-1.4,.1) -- (-1/2,1);
\draw [thick] [->] (2.5,1) -- (1/2,-1);
\draw [thick] [<-] (2.5,-1) -- (1.6,-.1);
\draw [thick] (1.4,.1) -- (1/2,1);
\node at (-3/2,-3/2) {$w=-1$};
\node at (3/2,-3/2) {$w=1$};
\end{tikzpicture}
\caption{The writhe of a crossing is the same if both orientations are reversed.}
\label{figeighteen}
\end{figure}
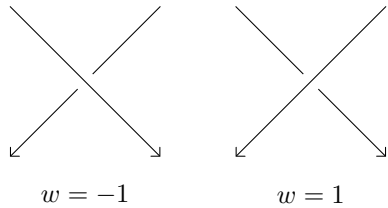

\begin{figure} [bht]
\centering
\begin{tikzpicture} [>=angle 90]
\draw [lightgray!50, fill=lightgray!50] (-2.5,1) -- (-1.5,0) -- (-.5,1);
\draw [lightgray!50, fill=lightgray!50] (-2.5,-1) -- (-1.5,0) -- (-1/2,-1);
\draw [lightgray!50, fill=lightgray!50] (2.5,1) -- (3/2,0) -- (.5,1);
\draw [lightgray!50, fill=lightgray!50] (5/2,-1) -- (3/2,0) -- (1/2,-1);
\draw [thick] [->] (-2.5,1) -- (-1/2,-1);
\draw [thick] [<-] (-2.5,-1) -- (-1.6,-.1);
\draw [thick] (-1.4,.1) -- (-1/2,1);
\draw [thick] [->] (2.5,1) -- (1/2,-1);
\draw [thick] [<-] (2.5,-1) -- (1.6,-.1);
\draw [thick] (1.4,.1) -- (1/2,1);
\node at (-3/2,-3/2) {$\eta_{or}=-1$};
\node at (-3/2,-2) {$\eta=-1,w=-1$};
\node at (3/2,-3/2) {$\eta_{or}=1$};
\node at (3/2,-2) {$\eta=1,w=1$};
\draw [lightgray!50, fill=lightgray!50] (6-2.5,1) -- (6-1.5,0) -- (6-.5,1);
\draw [lightgray!50, fill=lightgray!50] (6-2.5,-1) -- (6-1.5,0) -- (6-1/2,-1);
\draw [lightgray!50, fill=lightgray!50] (6+2.5,1) -- (6+3/2,0) -- (6+.5,1);
\draw [lightgray!50, fill=lightgray!50] (6+5/2,-1) -- (6+3/2,0) -- (6+1/2,-1);
\draw [thick] [->] (6-2.5,1) -- (6-1/2,-1);
\draw [thick] (6-2.5,-1) -- (6-1.6,-.1);
\draw [thick] [->] (6-1.4,.1) -- (6-1/2,1);
\draw [thick] [<-] (6+2.5,1) -- (6+1/2,-1);
\draw [thick] [<-] (6+2.5,-1) -- (6+1.6,-.1);
\draw [thick] (6+1.4,.1) -- (6+1/2,1);
\node at (6-3/2,-3/2) {$\eta_{or}=-\mathbf{i}$};
\node at (6-3/2,-2) {$\eta=-1,w=1$};
\node at (6+3/2,-3/2) {$\eta_{or}=\mathbf{i}$};
\node at (6+3/2,-2) {$\eta=1,w=-1$};
\end{tikzpicture}
\caption{Like the ordinary writhe, the checkerboard writhe $\eta_{or}$ is the same if both orientations are reversed.}
\label{fignineteen}
\end{figure}
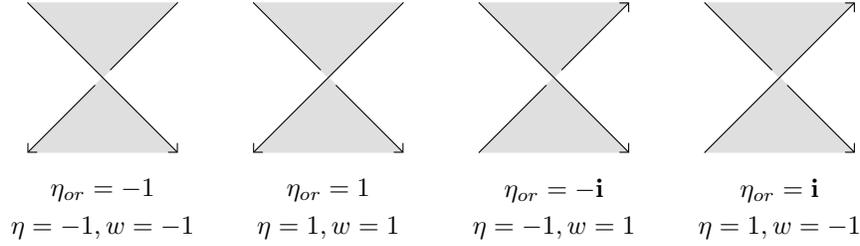

We use $\mathbf{i}=\sqrt{-1}$ in the checkerboard writhe simply to guarantee that $\eta_{or}$ values not related by underpass-overpass reversals are linearly independent over $\mathbb{Q}$. Any other element of $\mathbb{C}-\mathbb{Q}$ would serve the same purpose.

\begin{definition}
\label{goeritzmator}Let $D$ be an oriented link diagram, and $\sigma$ either of the two
checkerboard shadings of the complementary regions of $D$. Let $R_{1},\dots ,R_{n}$ be the
unshaded complementary regions according to $\sigma$, and for $i,j\in \{1,\dots
,n\}$ let $C_{ij}$ be the set of crossings of $D$ incident on $R_{i}$ and $R_{j}$. Then the \emph{oriented Goeritz matrix} of $D$ with respect to $\sigma$ is the $n\times n$ matrix $G_{or}(D,\sigma)$ with these entries: 
\[
G_{or}(D,\sigma)_{ij}=%
\begin{cases}
-\sum\limits_{c\in C_{ij}}\eta_{or}(c)\text{,} & \text{if }i\neq j\\
-\sum\limits_{k\neq i}G_{or}(D,\sigma)_{ik}\text{,} & \text{if }i=j
\end{cases}
\]
\end{definition}

\begin{definition}
\label{goeritzmator2}Let $\sigma$ be one of the two checkerboard shadings of the complementary regions of an oriented link diagram $D$, and let $B(D,\sigma)$ be the $(\beta_s(D,\sigma)-1) \times (\beta_s(D,\sigma)-1)$ matrix whose entries are all $0$. Then the \emph{adjusted, oriented} Goeritz matrix of $D$ with respect to $\sigma$ is 
\[
G_{or}^{adj}(D,\sigma)=
\begin{pmatrix}
G_{or}(D,\sigma) & 0 \\ 
0 & B(D,\sigma)
\end{pmatrix}.
\]
\end{definition}

Definition~\ref{goeritzmator2} yields an oriented version of Goeritz equivalence, in the same way that Definition~\ref{goeritzmat2} yields $\sim_G$.

\begin{definition}
\label{goeritzrelor}
Two oriented links $L$ and $L^{\prime }$ are \emph{Goeritz equivalent}, denoted $L\approx_{G} L'$, if there is a finite sequence $L=L_{1},\dots ,L_{k}=L^{\prime }$
of oriented links such that for each $i\in \{2,\dots ,k\}$, at least one of these conditions holds:
\begin{itemize} 
\item $[L_{i-1}]=[L_{i}]$, i.e., $L_{i-1}$ is ambient isotopic to $L_{i}$.
\item $L_{i-1}$ and $L_{i}$ have diagrams $D_{i-1}$ and $D_{i}$ with shadings $\sigma_{i-1}$ and $\sigma_{i}$ such that $G_{or}^{adj}(D_{i-1},\sigma_{i-1})=G_{or}^{adj}(D_{i},\sigma_{i})$.
\end{itemize}
\end{definition}

When appropriate, we also say that two oriented diagrams are Goeritz equivalent.

Here is a useful definition.

\begin{definition}
\label{detached}
Suppose $D$ is a diagram of $L$, $L_0$ is a sublink of $L$ and there is no crossing in $D$ involving an arc from $L_0$ and an arc from $L-L_0$. Then $L_0$ is a~\emph{detached sublink} in $D$. In particular, $L$ itself is a detached sublink in $D$.
\end{definition}

As noted in Figure~\ref{fignineteen}, reversing both orientations at a crossing does not affect the checkerboard writhe. Of course in practice we cannot reverse the orientations at just one crossing. However, if we reverse all the component orientations in a detached sublink then we obtain an oriented diagram $D'$ with $G_{or}^{adj}(D,\sigma) = G_{or}^{adj}(D',\sigma')$.

Adapting the discussion of mutation to oriented links is not quite so straightforward. Oriented elementary mutations are not uniquely defined, because an elementary mutation with respect to a tangle may require reversing all the arc orientations inside (or outside) the tangle. We deal with this ambiguity by including a new clause in the oriented version of Definition~\ref{mutant}. The preceding paragraph tells us that this new clause is implicit in Definition~\ref{goeritzrelor}.

\begin{definition}
\label{mutantor}Two oriented links $L$ and $L^{\prime }$ are \emph{Conway equivalent}, denoted $L\approx_{C} L'$, if there is a finite sequence $L=L_{1},\dots ,L_{k}=L^{\prime }$ of oriented links such that for each $i\in \{2,\dots ,k\}$, at least one of these conditions holds:
\begin{itemize} 
\item $[L_{i-1}]=[L_{i}]$, i.e., $L_{i-1}$ is ambient isotopic to $L_{i}$.
\item $L_{i-1}$ and $L_{i}$ have diagrams $D_{i-1}$ and $D_{i}$ such that $D_{i-1}$ is related to $D_{i}$ by an elementary mutation. If necessary, all the orientations within the tangle of the mutation are reversed.
\item $L_{i-1}$ and $L_{i}$ have diagrams $D_{i-1}$ and $D_{i}$ such that $D_{i}$ is obtained by reversing all orientations in a sublink that is detached in $D_{i-1}$.
\end{itemize}
\end{definition}

We are now ready to state oriented versions of Theorem~\ref{main} and Corollary~\ref{maincor}:

\begin{theorem}
\label{mainor}
If $L,L'$ are oriented links then $L\approx_C L'$ if and only if $L\approx_G L'$.
\end{theorem}

\begin{corollary}
\label{maincoror}
Let $\mathcal{I}$ be an invariant of classical, oriented link types. Then $\mathcal{I}$ is invariant under mutation if and only if for every choice of a shading $\sigma$ of a diagram $D$ of an oriented link $L$, $\mathcal{I}([L])$ is determined in some way by the adjusted, oriented Goeritz matrix $G_{or}^{adj}(D,\sigma)$.
\end{corollary}

We proceed to discuss the changes that must be made in the discussion of Sections 2 -- 7 to prove Theorem~\ref{mainor}. The only change required in Section 2 is to specify that the construction discussed there should not involve any change of orientations. Sections 3 and 4 also apply without any significant change; an elementary mutation does not affect the checkerboard writhe of any crossing even if some orientations are changed, because the two orientations at a crossing are either both preserved or both reversed. In particular we have the oriented version of Proposition~\ref{goeritzprop}:

\begin{proposition}
\label{goeritzpropor}
If $\sigma$ is a shading of an oriented link diagram $D$ then there is an oriented link diagram $D'$ with a strongly normal shading $\sigma'$, such that $D \approx_C D'$ and $G^{adj}_{or}(D,\sigma)=G^{adj}_{or}(D',\sigma')$.
\end{proposition}

An important difference between the oriented and unoriented cases is the fact that some symmetric matrices cannot be realized as $G_{or}^{adj}(D,\sigma)$ matrices. For instance, the matrix 

\[
\begin{pmatrix}
2 & -1 & -1 \\ 
-1 & 2 & -1 \\ 
-1 & -1 & 2%
\end{pmatrix}%
\]
is $G^{adj}(D,\sigma)$ for a trefoil diagram $D$, but it is not $G_{or}^{adj}(D,\sigma)$ for any oriented link diagram $D$. If it were, Proposition~\ref{goeritzpropor} would imply that the matrix could be realized by an oriented diagram with a strongly normal shading. As indicated in Figure~\ref{figtwenty}, though, no such diagram exists.

\begin{figure} [bht]
\centering
\begin{tikzpicture} [scale=0.7]
\draw [lightgray!50, fill=lightgray!50] (1,1) -- (2,0) -- (3,1) -- (3,1.5) -- (-3,1.5) -- (-3,1) -- (-2,0) -- (-1,1) -- (0,0);
\draw [lightgray!50, fill=lightgray!50] (1,-1) -- (2,0) -- (3,-1) -- (3,-1.5) -- (-3,-1.5) -- (-3,-1) -- (-2,0) -- (-1,-1) -- (0,0);
\draw [thick] (-3,1) -- (-2.1,.1);
\draw [thick] [->] [>=angle 90] (-1.9,-.1) -- (-1,-1);
\draw [thick] [->] [>=angle 90] (-1,-1) -- (1,1);
\draw [thick] (1,1) -- (1.9,.1);
\draw [thick] [->] [>=angle 90] (2.1,-.1) -- (3,-1);
\draw [thick] [<-] [>=angle 90] (-3,-1) -- (-1,1);
\draw [thick] [->] [>=angle 90] (-.1,.1) -- (-1,1);
\draw [thick] (.1,-.1) -- (1,-1);
\draw [thick] [->] [>=angle 90] (3,1) -- (1,-1);
\end{tikzpicture}
\caption{The diagram cannot be completed with only three unshaded regions, if the orientations are respected.}
\label{figtwenty}
\end{figure}
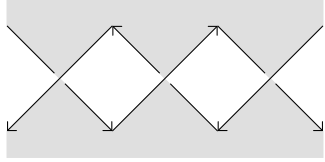

The changes in the discussion of Section 5 needed to reflect the issue of realizability are outlined below.

\begin{definition}
An \emph{oriented nonzero integer weighting} of a simple graph $\Gamma$ is a function $w:E(\Gamma) \to \{n,n \cdot \mathbf{i} \mid n \neq 0 \in \mathbb{Z}\}$. Each such weighting has an associated nonzero integer weighting $\overline{w}$ given by: if $w(e)=n \text{ or } w(e)=n \cdot \mathbf{i}$ then $\overline{w}(e)=n$. 
\end{definition}

\begin{definition}
Let $\Gamma$ be a planar simple graph with an oriented nonzero integer weighting $w$. Then $w$ is~\emph{realizable} if there is a local isotopy class $\lambda$ of imbeddings of $\Gamma$ in the plane, such that the components of the link diagrammed in $\mathcal{D}(\overline{w},\lambda)$ can be oriented so that for every edge $e \in E(\Gamma)$, every crossing corresponding to $e$ has checkerboard writhe $\eta_{or}=w(e)/|w(e)|$. The set of such ordered pairs $(w,\lambda)$ is denoted $\mathcal{IG}^{real}_{or}$.
\end{definition}

\begin{proposition} Suppose $\lambda$ and $\lambda'$ are local isotopy classes of imbeddings of the same simple graph $\Gamma$, and $w$ is an oriented nonzero integer weighting of $\Gamma$. Then $(w,\lambda) \in \mathcal{IG}^{real}_{or}$ if and only if $(w,\lambda') \in \mathcal{IG}^{real}_{or}$.
\end{proposition}
\begin{proof}
Suppose $(w,\lambda) \in \mathcal{IG}^{real}_{or}$. According to Proposition~\ref{goeritzpropor}, there is an oriented link diagram $D$ with a strongly normal shading, such that $D$ is an element of $\mathcal{D}(\overline{w},\lambda)$ and for each $e \in E(\Gamma)$, every crossing of $D$ corresponding to $e$ has checkerboard writhe $\eta_{or}=w(e)/|w(e)|$. The discussion of Sections 6 and 7 tells us that a diagram from $\mathcal{D}(\overline{w},\lambda)$ can be transformed into a diagram from $\mathcal{D}(\overline{w},\lambda')$  through ambient isotopies and elementary mutations. We can carry the component orientations from $D$ along while performing ambient isotopies and elementary mutations, up to the orientation ambiguity of mutation mentioned earlier; this ambiguity does not change the checkerboard writhe of any crossing, so it does not affect $w$. We conclude that $\mathcal{D}(\overline{w},\lambda')$ includes a diagram $D'$ that can be oriented so as to realize $w$.
\end{proof}

\begin{definition}
Let $\mathcal{IL}_{or}$ denote the set of local isotopy classes of oriented link diagrams with strongly normal shadings. 
\end{definition}

The proof of Theorem~\ref{corr} now yields the following.

\begin{theorem}
\label{corror}
There is a bijection $\mathcal{D}:\mathcal{IG}^{real}_{or} \to \mathcal{IL}_{or}$ with these properties:
\begin{itemize}
\item If $(w,\lambda) \in \mathcal{IG}^{real}_{or}$ and $D \in \mathcal{D}(w,\lambda)$ then $D$ is an oriented link diagram with a strongly normal shading $\sigma$ such that $\Gamma_u(D,\sigma)$ is isomorphic to the graph obtained from $\Gamma$ by replacing each edge $e \in E(\Gamma)$ with $|w(e)|$ parallel edges.
\item If $D \in \mathcal{D}(w,\lambda)$ and $e \in E(\Gamma)$ then all of the crossings in $D$ corresponding to $e$ have the same checkerboard writhe, $\eta_{or}=w(e)/|w(e)|$.
\end{itemize}
\end{theorem}

The proof of Theorem~\ref{mainor} is completed by following the argument of Section 7. 

\section{Two mutant knots}
\label{sec:cor2}

In this section and the next, we mention two interesting pairs of knots. Recall that for a knot, every shading $\sigma$ of a diagram $D$ has $\beta_s(D,\sigma)=1$, so we need not distinguish between $G(D,\sigma)$ and $G^{adj}(D,\sigma)$.

Stoimenow and his collaborators have carefully studied the Conway equivalences among thousands of classical knots of crossing number $\leq 16$; see for instance \cite{St, St1}. From this tremendous body of examples we require only two.

\begin{proposition}(\cite{St})
\label{stoim}
The knots $K_1=14_{46187}$ and $K_2=!14_{43907}$ pictured in Figure 1 of \cite{St} have the following properties.
\begin{enumerate}
    \item $[K_1] \neq [K_2]$.
    \item Each of $[K_1],[K_2]$ has a diagram with 14 crossings.
    \item No knot type Conway equivalent to $[K_1]$ or $[K_2]$ has a diagram with fewer than 14 crossings.
    \item The 14-crossing diagrams of $[K_1]$ and $[K_2]$ admit only elementary mutations that are trivial, i.e., they do not change the knot type.
    \item $[K_1]$ and $[K_2]$ have 15-crossing diagrams connected by an elementary mutation.
\end{enumerate}
\end{proposition}

Suppose we apply the proof of Proposition \ref{goeritzprop} to a 14-crossing diagram $D$ of $[K_1]$ or $[K_2]$. Every step of the argument is either an ambient isotopy or an elementary mutation, and no step increases the number of crossings. Items 3 and 4 of Proposition \ref{stoim} guarantee that no step of the proof has any effect on the knot type. It follows that if $\sigma$ is a shading of a 14-crossing diagram $D$ of $[K_1]$ (or $[K_2]$), then there is a 14-crossing diagram $D'$ of $[K_1]$ (or $[K_2]$) with a strongly normal shading $\sigma'$, such that $G(D,\sigma)=G(D',\sigma')$.

\begin{lemma}
\label{nog}
If $\sigma_1$ and $\sigma_2$ are shadings of 14-crossing diagrams $D_1$ and $D_2$ of $[K_1]$ and $[K_2]$, then $G(D_1,\sigma_1) \neq G(D_2,\sigma_2)$.
\end{lemma}

\begin{proof}
Suppose instead that $G(D_1,\sigma_1) = G(D_2,\sigma_2)$. As was just noted, we may presume that $\sigma_1$ and $\sigma_2$ are strongly normal shadings of $D_1$ and $D_2$. 

The equality $G(D_1,\sigma_1) = G(D_2,\sigma_2)$ implies that there is an isomorphism between the graphs $\Gamma_u(D_1,\sigma_1)$ and $\Gamma_u(D_2,\sigma_2)$, such that edges matched by the isomorphism correspond to crossings of the same Goeritz index. Let $\Gamma$ be the simple graph obtained from $\Gamma_u(D_1,\sigma_1)$ and $\Gamma_u(D_2,\sigma_2)$ by combining each set of parallel edges into one edge, and let $w$ be the nonzero integer weighting of $\Gamma$ corresponding to $\Gamma_u(D_1,\sigma_1)$ and $\Gamma_u(D_2,\sigma_2)$. (The equality $G(D_1,\sigma_1) = G(D_2,\sigma_2)$ implies that $\Gamma_u(D_1,\sigma_1)$ and $\Gamma_u(D_2,\sigma_2)$ are described by the same integer weighting of $\Gamma$.) Then there are local isotopy classes $\lambda_1, \lambda_2$ of imbeddings of $\Gamma$ in the plane, such that $D_1$ and $D_2$ are the images of $(w,\lambda_1)$ and $(w,\lambda_2)$ under the bijection of Theorem \ref{corr}.

There is a sequence of the operations of Section \ref{sec:imbed}, which transforms $\lambda_1$ into $\lambda_2$. As discussed in Section \ref{sec:gc}, these operations induce a sequence of local isotopies and elementary mutations, which preserve the number of crossings and transform $D_1$ into $D_2$. According to Proposition \ref{stoim}, these induced operations all preserve the knot type. But this is impossible, as $[K_1] \neq [K_2]$. \end{proof}

Proposition \ref{nog} implies that $\mathcal G([K_1]) \neq \mathcal G([K_2])$. On the other hand, the 15-crossing diagrams $D'_1,D'_2$ of $[K_1]$ and $[K_2]$ given in \cite[Figure 1]{St} are related through an elementary mutation of the second kind, and they share the Goeritz matrix $G(D'_1,\sigma_u)=G(D'_2,\sigma_u)$. Therefore  $K_1$ and $K_2$ share some of their (adjusted) Goeritz matrices, but not all.

\section{Two non-mutant knots}
\label{sec:cforms}

Goeritz~\cite{G} showed that some Reidemeister moves have the effect of replacing a Goeritz matrix $G(D,\sigma)$ with a congruent matrix, i.e., a matrix of the form $UG(D,\sigma)U^T$ where $U$ is a unimodular matrix of integers. The equivalence class of a symmetric integer matrix under unimodular congruence is the set of integer matrices that represent a single bilinear or quadratic form, with respect to different bases. Consequently it is natural to think that the knot-theoretic significance of the Goeritz matrix is tied to the quadratic form (or bilinear form) which the matrix represents. Many authors who have written about Goeritz matrices have explicitly discussed forms and their associated invariants, like the determinant, the $d$-invariant and the Minkoski units. Examples include Crowell~\cite{C}, Goeritz himself~\cite{G}, Gordon and Litherland \cite{GL}, Greene \cite{Gre}, Kyle \cite{K} and Lee \cite{Lee}. Other authors have not mentioned forms but instead have focused on congruence of the Goeritz matrix, like Im, Lee and Lee~\cite{ILL} and Traldi \cite{T}. 

It is important to realize that Theorem~\ref{main} does not follow the same theme. Rather than allowing arbitrary congruences of (adjusted) Goeritz matrices, Definition~\ref{goeritzrel} takes into account only diagram changes under which either the link type or the adjusted Goeritz matrix does not change at all. We illustrate this point with an example.

\begin{figure} [bht]
\centering
\begin{tikzpicture} [scale=.3]
\draw [lightgray!50, fill=lightgray!50] (-18,4) -- (-18,-6) -- (-20,-6) -- (-20,4);
\draw [lightgray!50, fill=lightgray!50] (-18,2) -- (-17,3) -- (-18,4);
\draw [lightgray!50, fill=lightgray!50] (-18,2) -- (-15,-1) -- (-18,-4);
\draw [lightgray!50, fill=lightgray!50] (-19,0) -- (-18,-6) -- (-15,-6) -- (-14,-5);
\draw [lightgray!50, fill=lightgray!50] (8,4) -- (8,-6) -- (6,-6) -- (6,4);
\draw [lightgray!50, fill=lightgray!50] (6,2) -- (5,3) -- (6,4);
\draw [lightgray!50, fill=lightgray!50] (6,2) -- (3,-1) -- (6,-4);
\draw [lightgray!50, fill=lightgray!50] (7,0) -- (6,-6) -- (3,-6) -- (2,-5);
\draw [lightgray!50, fill=lightgray!50] (-16,2) -- (-17,3) -- (-16,4) -- (-15,3);
\draw [lightgray!50, fill=lightgray!50] (-14,2) -- (-15,3) -- (-14,4) -- (-13,3);
\draw [lightgray!50, fill=lightgray!50] (-12,2) -- (-13,3) -- (-12,4) -- (-11,3);
\draw [lightgray!50, fill=lightgray!50] (-10,2) -- (-11,3) -- (-10,4) -- (-9,3);
\draw [lightgray!50, fill=lightgray!50] (4,2) -- (5,3) -- (4,4) -- (3,3);
\draw [lightgray!50, fill=lightgray!50] (2,2) -- (3,3) -- (2,4) -- (1,3);
\draw [lightgray!50, fill=lightgray!50] (0,2) -- (1,3) -- (0,4) -- (-1,3);
\draw [lightgray!50, fill=lightgray!50] (-2,2) -- (-1,3) -- (-2,4) -- (-3,3);
\draw [lightgray!50, fill=lightgray!50] (-9,3) -- (-8,4) -- (-4,4) -- (-3,3);
\draw [lightgray!50, fill=lightgray!50] (-9,3) -- (-8,2) -- (-4,2) -- (-3,3);
\draw [lightgray!50, fill=lightgray!50] (-11,-1) -- (-1,-1);
\draw [lightgray!50, fill=lightgray!50] (-11,-1) -- (-8,2) -- (-4,2) -- (-1,-1);
\draw [lightgray!50, fill=lightgray!50] (-11,-1) -- (-9.5,-2.5) -- (-2.5,-2.5) -- (-1,-1);
\draw [lightgray!50, fill=lightgray!50] (-9.5,-2.5) -- (-2.5,-2.5) -- (0,-5) -- (-1,-6) -- (-11,-6) -- (-12,-5);
\draw [lightgray!50, fill=lightgray!50] (-12,-5) -- (-13,-6) -- (-14,-5) -- (-13,-4);
\draw [lightgray!50, fill=lightgray!50] (-12,-2) -- (-11,-1) -- (-12,0) -- (-13,-1);
\draw [lightgray!50, fill=lightgray!50] (-14,-2) -- (-13,-1) -- (-14,0) -- (-15,-1);
\draw [lightgray!50, fill=lightgray!50] (2,-2) -- (1,-1) -- (2,0) -- (3,-1);
\draw [lightgray!50, fill=lightgray!50] (0,-2) -- (-1,-1) -- (0,0) -- (1,-1);
\draw [lightgray!50, fill=lightgray!50] (1,-6) -- (2,-5) -- (1,-4) -- (0,-5);
\draw [lightgray!50, fill=lightgray!50] (13,-6) -- (14,-6) -- (14,4) -- (13,4);
\draw [lightgray!50, fill=lightgray!50] (14,4) -- (20,4) -- (20,3) -- (19,2) -- (14,2);
\draw [lightgray!50, fill=lightgray!50] (20,3) -- (18,1) -- (16,3);
\draw [lightgray!50, fill=lightgray!50] (17,2) -- (16,1) -- (15,2);
\draw [lightgray!50, fill=lightgray!50] (13,2) -- (14,1) -- (15,2);
\draw [lightgray!50, fill=lightgray!50] (14,-3) -- (15,-4) -- (16,-3) -- (17,-4) -- (18,-3) -- (20,-5) -- (20,-6) -- (14,-6);
\draw [lightgray!50, fill=lightgray!50] (14,-3) -- (15,-2) -- (16,-3) -- (17,-2) -- (18,-3) -- (19,-2) -- (19,0) -- (18,1) -- (17,0) -- (16,1) -- (15,0) -- (14,1);
\draw [thick] (-18,4) -- (-16,2);
\draw [thick] (-16,4) -- (-14,2);
\draw [thick] (-14,4) -- (-12,2);
\draw [thick] (-12,4) -- (-10,2);
\draw [thick] (-10,4) -- (-8,2);
\draw [thick] (-18,2) -- (-17.2,2.8);
\draw [thick] (-16.8,3.2) -- (-16,4);
\draw [thick] (-16,2) -- (-15.2,2.8);
\draw [thick] (-14.8,3.2) -- (-14,4);
\draw [thick] (-14,2) -- (-13.2,2.8);
\draw [thick] (-12.8,3.2) -- (-12,4);
\draw [thick] (-12,2) -- (-11.2,2.8);
\draw [thick] (-10.8,3.2) -- (-10,4);
\draw [thick] (-10,2) -- (-9.2,2.8);
\draw [thick] (-8.8,3.2) -- (-8,4);
\draw [thick] (-16,0) -- (-14,-2);
\draw [thick] (-14,0) -- (-12,-2);
\draw [thick] (-12,0) -- (-9.5,-2.5);
\draw [thick] (-16.5,-2.5) -- (-15.2,-1.2);
\draw [thick] (-14.8,-0.8) -- (-14,0);
\draw [thick] (-14,-2) -- (-13.2,-1.2);
\draw [thick] (-12.8,-0.8) -- (-12,0);
\draw [thick] (-12,-2) -- (-11.2,-1.2);
\draw [thick] (-10.8,-0.8) -- (-10,0);
\draw [thick] (-8,2) -- (-10,0);
\draw [thick] (-18,2) -- (-16,0);
\draw [thick] (-18,4) -- (-20,4);
\draw [thick] (-20,-6) -- (-20,4);
\draw [thick] (-20,-6) -- (-15,-6);
\draw [thick] (-16.5,-2.5) -- (-15,-4);
\draw [thick] (-9.5,-2.5) -- (-11,-4);
\draw [thick] (-15,-6) -- (-13,-4);
\draw [thick] (-13,-6) -- (-11,-4);
\draw [thick] (-15,-4) -- (-14.2,-4.8);
\draw [thick] (-13.8,-5.2) -- (-13,-6);
\draw [thick] (-13,-4) -- (-12.2,-4.8);
\draw [thick] (-11.8,-5.2) -- (-11,-6);
\draw [thick] (-11,-6) -- (-1,-6);
\draw [thick] (-8,4) -- (-4,4);
\draw [thick] (-4,2) -- (-2,4);
\draw [thick] (-2,2) -- (0,4);
\draw [thick] (0,2) -- (2,4);
\draw [thick] (2,2) -- (4,4);
\draw [thick] (4,2) -- (6,4);
\draw [thick] (-4,4) -- (-3.2,3.2);
\draw [thick] (-2.8,2.8) -- (-2,2);
\draw [thick] (-2,4) -- (-1.2,3.2);
\draw [thick] (-0.8,2.8) -- (0,2);
\draw [thick] (0,4) -- (0.8,3.2);
\draw [thick] (1.2,2.8) -- (2,2);
\draw [thick] (2,4) -- (2.8,3.2);
\draw [thick] (3.2,2.8) -- (4,2);
\draw [thick] (4,4) -- (4.8,3.2);
\draw [thick] (5.2,2.8) -- (6,2);
\draw [thick] (-2.5,-2.5) -- (0,0);
\draw [thick] (0,0) -- (0.8,-0.8);
\draw [thick] (1.2,-1.2) -- (2,-2);
\draw [thick] (0,-2) -- (2,0);
\draw [thick] (0,-2) -- (-0.8,-1.2);
\draw [thick] (-4,2) -- (-1.2,-0.8);
\draw [thick] (2,0) -- (2.8,-0.8);
\draw [thick] (3.2,-1.2) -- (4,-2);
\draw [thick] (2,-2) -- (6,2);
\draw [thick] (-2.5,-2.5) -- (1,-6);
\draw [thick] (1,-4) -- (3,-6);
\draw [thick] (1,-4) -- (0.2,-4.8);
\draw [thick] (-1,-6) -- (-0.2,-5.2);
\draw [thick] (1,-6) -- (1.8,-5.2);
\draw [thick] (2.2,-4.8) -- (4.5,-2.5);
\draw [thick] (4,-2) -- (4.5,-2.5);
\draw [thick] (3,-6) -- (8,-6);
\draw [thick] (6,4) -- (8,4);
\draw [thick] (8,-6) -- (8,4);
\draw [thick] (13,-6) -- (13,4);
\draw [thick] (20,-6) -- (13,-6);
\draw [thick] (20,4) -- (13,4);
\draw [thick] (20,4) -- (20,3);
\draw [thick] (18.2,-3.2) -- (20,-5);
\draw [thick] (20,-6) -- (20,-5);
\draw [thick] (17.8,-2.8) -- (17,-2);
\draw [thick] (19,-2) -- (17,-4);
\draw [thick] (15,-2) -- (17,-4);
\draw [thick] (15,-2) -- (14,-3);
\draw [thick] (14,-3) -- (15,-4);
\draw [thick] (15.8,-3.2) -- (15,-4);
\draw [thick] (16.2,-2.8) -- (17,-2);
\draw [thick] (19,0) -- (19,-2);
\draw [thick] (19,0) -- (17,2);
\draw [thick] (15,0) -- (17,2);
\draw [thick] (15,0) -- (14,1);
\draw [thick] (15,2) -- (14,1);
\draw [thick] (15,2) -- (15.8,1.2);
\draw [thick] (17,0) -- (16.2,0.8);
\draw [thick] (17,0) -- (17.8,0.8);
\draw [thick] (18.2,1.2) -- (20,3);
\node at (10.5,-1) {1};
\node at (17,-3) {5};
\node at (15,-3) {4};
\node at (17,1) {3};
\node at (15,1) {2};
\node at (1,1) {5};
\node at (1,-3) {4};
\node at (-13,1) {3};
\node at (-13,-3) {2};
\end{tikzpicture}
\caption{On the left, the connected sum of a $(5,3,-2)$ pretzel knot and a $(-5,-3,2)$ pretzel knot. On the right, an unknot. Unshaded regions are numbered to indicate row and column indices in the associated Goeritz matrices.} 
\label{figtwentyone}
\end{figure}
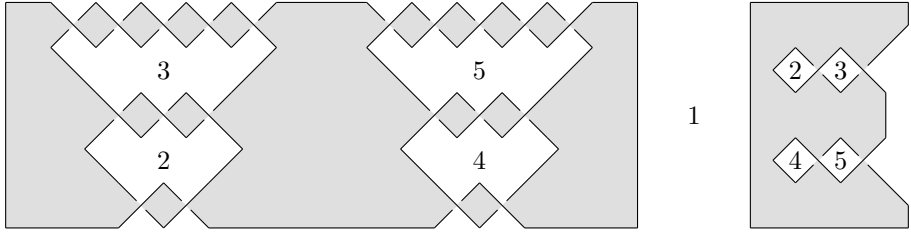

Two knot diagrams are given in Figure \ref{figtwentyone}. The diagram on the left represents a connected sum of two pretzel knots, and the diagram on the right represents a trivial knot. The (adjusted) Goeritz matrices associated with the indicated shadings are congruent: 
\[
\begin{pmatrix}
1 & 2 & 0 & 2 & 0 \\ 
0 & 1 & 0 & 0 & 0 \\ 
0 & -2 & 1 & 0 & 0 \\ 
0 & 0 & 0 & 1 & 0 \\ 
0 & 0 & 0 & -2 & 1%
\end{pmatrix}%
\begin{pmatrix}
0 & 0 & 1 & 0 & -1 \\ 
0 & 1 & -1 & 0 & 0 \\ 
1 & -1 & 0 & 0 & 0 \\ 
0 & 0 & 0 & -1 & 1 \\ 
-1 & 0 & 0 & 1 & 0%
\end{pmatrix}%
\begin{pmatrix}
1 & 0 & 0 & 0 & 0 \\ 
2 & 1 & -2 & 0 & 0 \\ 
0 & 0 & 1 & 0 & 0 \\ 
2 & 0 & 0 & 1 & -2 \\ 
0 & 0 & 0 & 0 & 1%
\end{pmatrix}
\]
\[
=\begin{pmatrix}
0 & 2 & -5 & -2 & 5 \\ 
2 & 1 & -3 & 0 & 0 \\ 
-5 & -3 & 8 & 0 & 0 \\ 
-2 & 0 & 0 & -1 & 3 \\ 
5 & 0 & 0 & 3 & -8%
\end{pmatrix}
\]

The two knots are not equivalent under mutation, however. (It is well known that a nontrivial knot cannot be equivalent to an unknot under mutation; see~\cite{M} for instance.) Consequently Theorem \ref{main} tells us that the two diagrams are not Goeritz equivalent.

By the way, the same point can be made using only one of the pretzel knots, rather than their connected sum. We provide the pictured example because these two Goeritz matrices are not only congruent; they are also equivalent under the ``corrected'' Goeritz matrix definition proposed in \cite{T}. 

It is no surprise that the diagrams in Figure~\ref{figtwentyone} are not alternating. For according to a theorem of Greene~\cite{Gre}, reduced alternating diagrams with congruent Goeritz matrices are Conway equivalent. (It has been known for a long time that the form represented by the Goeritz matrix is particularly useful for alternating links; see~\cite{C} for instance.)

\section{Questions}

We close with four interesting questions involving the ideas of the paper.

1. If $K_1$ and $K_2$ are Conway equivalent knots that are not ambient isotopic, then a split link consisting of $\mu$ copies of $K_1$ will require at least $\mu$ elementary mutations to be transformed into a split link consisting of $\mu$ copies of $K_2$. It follows that there is no bound on $k$ in Definition \ref{mutant}. Is there a bound on $k$ in Definition \ref{goeritzrel}?

2. In particular, is it possible for two link types to require $k>1$ in Definition \ref{goeritzrel}? This would mean that $[L_1]$ and $[L_2]$ are Goeritz equivalent but do not share any of their adjusted Goeritz matrices.

3. At the other extreme: what does it mean for two link types to share all of their adjusted Goeritz matrices?

4. According to Corollaries \ref{maincor} and \ref{maincoror}, every mutation invariant $\mathcal{I}([L])$ of an (un)oriented classical link type is determined in some way by each matrix $G^{adj}(D,\sigma)$ or $G^{adj}_{or}(D,\sigma)$ associated to a diagram of $[L]$. For each such $\mathcal{I}([L])$, is it possible to describe precisely how the matrix determines the invariant?

As an instance of question 4, consider the simplest link invariant: the number of link components, $\mu(L)$. It is an elementary exercise to verify that $\mu(L)$ is invariant under mutations. (On the left-hand side of Figure \ref{figtwo}, the four indicated segments must be connected twice in pairs, once inside the tangle and once outside the tangle. In each of the nine resulting cases, it is easy to see that the number of link components is preserved by all three elementary mutations.) According to Corollary \ref{maincor} it follows that $\mu(L)$ is connected to $G^{adj}(D,\sigma)$ in some way. The connection turns out to be this: if we reduce the entries of $G^{adj}(D,\sigma)$ modulo 2, then $\mu(L)$ is the nullity of the resulting matrix over $GF(2)$, the field of two elements. (Silver and Williams gave a short proof of the special case of this nullity formula for connected link diagrams in~\cite{SW}; they described the nullity formula as ``part of topology's folklore.'')

\bibliographystyle{plain}
\bibliography{mutgoer}

\end{document}